\documentclass{scrartcl}
\title{Part I, Free Actions of Compact \\ Abelian Groups on C\Star Algebras}
\author{Kay Schwieger and Stefan Wagner}
\date{}

\KOMAoptions{parskip=half, paper=a4, abstract=on}
\usepackage[T1]{fontenc}
\usepackage{lmodern}
\usepackage{microtype}		
\usepackage[english]{babel}
\usepackage{amsmath, amsthm, amssymb, bbm, mathtools, nicefrac}
\usepackage{enumitem}
	\newlist{equivalence}{enumerate}{1}
	\setlist[equivalence]{label=(\alph*)}
\usepackage{xspace}
\usepackage{hyperref}

\theoremstyle{plain}
	\newtheorem{thm}{Theorem}[section]
	\newtheorem{prop}[thm]{Proposition}
	\newtheorem{lemma}[thm]{Lemma}
	\newtheorem{cor}[thm]{Corollary}

\theoremstyle{definition}
	\newtheorem{defn}[thm]{Definition}
	\newtheorem{rmk}[thm]{Remark}
	\newtheorem{expl}[thm]{Example}
\theoremstyle{remark}

\newcommand*{\one}{\mathbbm 1}		
\newcommand*{\tensor}{\otimes}		
	
\DeclarePairedDelimiter{\scal}{\langle}{\rangle}	
\DeclarePairedDelimiter{\norm}{\lVert}{\rVert} 	
\DeclareMathOperator{\id}{id}		

\DeclareMathOperator{\Aut}{Aut}
\DeclareMathOperator{\Inn}{Inn}

\DeclareMathOperator{\M}{M}

\DeclareMathOperator{\Hom}{Hom}
\DeclareMathOperator{\Homeo}{Homeo}
\DeclareMathOperator{\Out}{Out}

\DeclareMathOperator{\SU}{SU}

\DeclareMathOperator{\pr}{pr}

\DeclareMathOperator{\Alt}{Alt}

\DeclareMathOperator{\Ext}{Ext}

\DeclareMathOperator{\tr}{tr}

\DeclareMathOperator{\Span}{span}

\DeclareMathOperator{\Pic}{Pic}

\newcommand{\cf}{\mbox{cf.}\xspace}			
\newcommand{\eg}{\mbox{e.\,g.}\xspace}			
\newcommand*{\ie}{\mbox{i.\,e.}\xspace}			
\newcommand*{\ndash}{\nobreakdash-}

\newcommand*{\alg}{\mathcal}		
\newcommand*{\End}{\mathcal L}		
\newcommand*{\Star}{$^*$\ndash}
\newcommand*{\aA}{\alg A}		
\newcommand*{\aB}{\alg B}
\newcommand*{\aC}{\alg C}
\DeclareMathOperator{\Ad}{Ad}		

\begin{document}

\author{
	Kay Schwieger \thanks{
		University of Helsinki, 
		\href{mailto:kay.schwieger@gmail.com}{\nolinkurl{kay.schwieger@gmail.com}}
	} \and 
	Stefan Wagner \thanks{
		Blekinge Tekniska H\"ogskola,
		\href{mailto:stefan.wagner@bth.se}{\nolinkurl{stefan.wagner@bth.se}}
	}
}
\maketitle

\begin{abstract}
	\noindent
	We study free actions of compact groups on unital C\Star algebras. In particular, we provide a complete classification theory of these actions for compact Abelian groups and explain its relation to the classification of classical principal bundles.

	\vspace*{0,5cm}


\end{abstract}

\section{Introduction}

In this article we study free actions of compact groups on unital C\Star algebras. This class of actions was first introduced under the name \emph{saturated actions} by Rieffel \cite{Ri88} and equivalent characterizations where given by Ellwood \cite{Ellwood00} and by Gottman, Lazar, and Peligrad \cite{GoLaPe94} (see also \cite{DKY12,Wa10}). Other related notions of freeness were studied by Phillips \cite{Ph87} in connection with K-theory (see also \cite{Ph09}).

Free actions do not admit degeneracies that may be present in general actions. For this reasons they are easier to understand and to classify.  In fact, for compact Abelian groups, free ergodic actions, \ie, free actions with trivial fixed point algebra, were completely classified by Olesen, Pedersen and Takesaki in \cite{OlPeTa80} and independently by Albeverio and H\o egh--Krohn in \cite{AlHo80}. This classification was generalized to compact non-Abelian groups by the remarkable work of Wassermann \cite{Wass88a, Wass88b, Wass89}. According to \cite{AlHo80,OlPeTa80,Wass88a}, for a compact group $G$ there is a 1-to-1 correspondence between free ergodic actions of $G$ and 2-cocycles of the dual group. An analogous result in the context of compact quantum groups has been obtained by Bichon, De Rijdt and Vaes \cite{BiRiVa06}. Extending these results beyond the ergodic case is however not straightforward because, even for a commutative fixed point algebra, the action cannot necessarily be decomposed into a bundle of ergodic actions. For compact Abelian groups our results about free, but not necessary ergodic actions may be regarded as a generalization of the classification given in \cite{AlHo80,OlPeTa80}. We would also like to point out that Neshveyev \cite{Ne13} obtained a classification of actions of compact quantum groups in terms of weak unitary tensor functors, which do unfortunatley not have an obvious homological interpretation. 

The study of non-ergodic free actions is also motivated by the established theory of principal bundles. In fact, by a classical result, having a free action of a compact group $G$ on a compact $P$ is equivalent saying that $P$ carries the structure of a principal bundle over the quotient $X := P/G$ with structure group~$G$. Very well-understood is the case of locally trivial principal bundles, that is, if $P$ is glued together from spaces of the form $U \times G$ with an open subset $U \subseteq X$. This gluing immediatly leads to $G$-valued cocycles. The corresponding cohomology theory, called \v{C}ech cohomology, gives a complete classification of locally trivial principal bundles with base space $X$ and structure group~$G$ (see \cite{To00}). For principal bundles that are not locally trivial, however, there is no obvious classification available. Our results provide such a classification in the case of a compact Abelian structure group.

Passing to noncommutative geometry poses the question how to extend the concept of principal bundles to noncommutative geometry.  In the case of vector  bundles the Theorem of Serre and Swan (\cf \cite{Swa62}) gives the essential clue: The category of vector bundles over a compact space $X$ is equivalent to the category of finitely generated and projective $C(X)$-modules. This observation leads to a notion of noncommutative vector bundles and is the connection between the topological K-theory based on vector bundles and the K-theory for C\Star algebras. For principal bundles, free and proper actions offer a good candidate for a notion of noncommutative principal bundles (see \eg  \cite{BHS07, BDH13, Ellwood00, Ph09}). A~similar geometric approach based on transformation groups was developed by one of the authors \cite{Wa11d, Wa11e}. In a purely algebraic setting, the well-established theory of Hopf--Galois extensions provides a wider framework comprising coactions of Hopf algebras (\eg \cite{Haj04,LaSu05,Sch04}). We also would like to mention the related notion of noncommutative principal torus bundles proposed by Echterhoff, Nest, and Oyono-Oyono \cite{ENOO09} (see also \cite{HaMa10}), which relies on a noncommutative version of Green's Theorem. Considering free actions as noncommutative principal bundles, the present article characterizes principal bundles in terms of associated vector bundles. 


Extending the classical theory of principal bundles to noncommutative geometry is not of purely mathematical interest.  In fact, noncommutative principal bundles become more and more prevalent in geometry and physics. For instance, Ammann and B\"ar \cite{Amm, AmmB} study the properties of the Riemannian spin geometry of a smooth principal $U(1)$\ndash bundle. Under suitable hypotheses, they relate the spin structure and the Dirac operator on the total space to the spin structure and the Dirac operator on the base space.
A~noncommutative generalization of these results was developed by Dabrowski, Sitarz, and Zucca in \cite{DaSi,DaSiZu} using spectral triples and the Hopf--Galois analogue of principal $U(1)$-bundles. Noncommutative principal bundles also appear in the study of 3-dimensional topological quantum field theories that are based on the modular tensor category of representations of the Drinfeld double (\cf \cite{LePr15}). In this context, special types of Hopf--Galois extensions correspond to symmetries of the theory or, equivalently, to invertible defects. As such, they are connected to module categories and, in particular, to the Brauer--Picard group of pointed fusion categories. Furthermore, T-duality is considered to be an important symmetry of string theories (\cite{AAGBL94, Bu87}). It is known that a circle bundle with H-flux given by a Neveu--Schwarz 3\ndash form admits a T-dual circle bundle with dual H-flux. However, it is also known that in general torus bundles with H-flux do not necessarily have a T-dual that is itself a classical torus bundle. Mathai and Rosenberg showed in \cite{MaRo05, MaRo06} that this problem is resolved by passing to noncommutative spaces. For example, it turns out that every principal $\mathbb{T}^2$\ndash bundle with H-flux does indeed admit a T-dual but its T-dual is non-classical. It is a bundle of noncommutative 2-tori, which can (locally) be realized as a noncommutative principal $\mathbb{T}^2$-bundle in the sense of \cite{ENOO09}. All these examples demand a better understanding of the geometry of noncommutative principal bundles. Although the classification in this article relies purely on the topology of the bundle, we hope that parts of our classification extends in such a way that additional geometrical information of the space is comprised. 

In the present article we investigate the structure of free actions on C\Star algebras, which is one framework for noncommutative principal bundles. We restrict ourselves to the compact setting, that is, to compact groups acting on unital C\Star algebras. The main objective of this article is to provide a complete classification of free actions of compact Abelian groups on unital C\Star algebras. We achieve such a classification by inspecting the Morita equivalence bimodule structure of the isotypic components. It turns out that the resulting classification can be handled by methods of group cohomology. More detailedly, the paper is organized as follows.

After some preliminaries, we discuss the different equivalent characterizations of freeness in the first part of Section~\ref{sec:general_th} (Theorem~\ref{equcondsatact}). In the second part of Section~\ref{sec:general_th} we further provide some methods to construct new free actions from given ones. As an example we present a one-parameter family of free $\SU(2)$-actions which is related to the Connes-Landi spheres (\cf~\cite{LaSu05}).


Section~\ref{prelud} is the first essential step in our classification. We construct a first invariant of a C\Star dynamical system, namely a group homomorphism $\varphi:\widehat{G} \to \Pic(\aB)$ from the dual of the compact Abelian group $G$ to the Picard group of the unital fixed point algebra~$\aB$. In general, this invariant neither distinguishes all free actions of $G$ nor is there a dynamical system for each group homomorphism $\varphi$. However, refining this invariant leads to our classifying data, to which we refer as \emph{factor system} due to some similarity with the theory of group extensions. As one cornerstone of our classification, we show that every factor system can indeed be obtained from a C\Star dynamical system. The construction of this dynamical system forms the main part of Section~\ref{prelud}.

The full classification of free actions of compact Abelian groups on unital C\Star algebras is discussed in Section~\ref{sec:classification}. For this purpose we additionally fix a group homomorphism $\varphi: \widehat{G} \to \Pic(\aB)$ and restrict our attention to C\Star dynamical system with the given $\varphi$ as invariant. The main result is that, if such dynamical systems exist, all free actions associated to the triple $(\aB, G, \varphi)$ are parametrized, up to 2-coboundaries, by 2-cocycles on the dual group $\widehat{G}$ with values in the group $UZ(\aB)$ of central unitary elements in~$\aB$ (Theorem~\ref{sat and cocycles} and Corollary~\ref{sat and cocycles cor}). In other words, the set in question is a principal homogenous space with respect to a classical cohomology group $H^2(\widehat{G},UZ(\aB))$. In the remaining part of this section we provide a group theoretic criterion for the existence of free actions with invariant $\varphi$; that is, factor systems associated to the triple~$(\aB,G,\varphi)$.

As already mentioned, for a compact group $G$ and a compact space $X$, locally trivial principal $G$-bundles over $X$ are classified by the \v Cech cohomology for the pair $(X,G)$. From the C\Star algebraic viewpoint, local triviality is not easy to capture. Section~\ref{sec:principal_bdl} provides instead a classification of not necessarily locally trivial principal bundles in case of a compact Abelian structure group. 
Finally, in Section~\ref{sec:examples} we discuss a few examples.

We would like to point out that with little effort the arguments and the results which are presented in this article for actions of compact Abelian groups extend to coactions of group C\Star algebras of finite groups. We also would like to mention that this article is the first part of a larger program aiming at classifying more general free actions on C\Star algebras (\cf \cite{SchWa15,SchWa16a}). This classification may be used to develop a fundamental group for C\Star algebras (\cf \cite{SchWa16b}). It could also serve as a starting point for an approach to quantum gerbes and a theory of T-duals for noncommutative principal torus bundles. 


\section{Preliminaries and Notations}
\label{sec:prel}

Our study is concerned with free actions of compact groups on unital C\Star algebras and their classification. As a consequence, we use and blend tools from geometry, representation theory and operator algebras. In this preliminary section we provide the most important definitions and notations which are repeatedly used in this article. 


\subsubsection*{Principal Bundles and Free Group Actions}

Let $P$ and $X$ be compact spaces. Furthermore, let $G$ be a compact group. A~locally trivial principal bundle is a quintuple $(P,X,G,q,\sigma)$, where $q:P\rightarrow X$ is a continuous map and $\sigma:P\times G\rightarrow P$ a continuous action, with the property of local triviality: Each point $x\in X$ has an open neighbourhood $U$ for which there exists a homeomorphism $\varphi_U:U\times G\rightarrow q^{-1}(U)$ satisfying $q\circ\varphi_U=\pr_U$ and additionally the equivariance property 
\begin{align*}
\varphi_U(x,gh)=\varphi_U(x,g).h
\end{align*}
for $x\in U$ and $g,h\in G$. It follows that the map $q$ is surjective, that the action $\sigma$ is free and proper, and that the natural map $P/G\mapsto X$, $p.G\mapsto q(p)$ is a homeomorphism. In particular, we recall that the action $\sigma$ is called free if and only if all stabilizer groups $G_p:=\{g\in G\;|\;\sigma(p,g)=p\}$, $p\in P$, are trivial. For a solid background on free group actions and principal bundles we refer to \cite{Hu75,KoNo63,Ne08b}.

\subsubsection*{C\Star Dynamical Systems}

Tensor products of C\Star algebras are taken with respect to the minimal tensor product denoted by $\tensor$. Let $\aA$ be a unital C\Star algebra and $G$ a compact group that acts on $\aA$ by \Star automorphism $\alpha_g:\aA \to \aA$ ($g \in G$) such that the map $G \times \aA \to \aA$, $(g,a) \mapsto \alpha_g(a)$ is continuous. Throughout this paper we call such a triple $(\aA,G,\alpha)$ a C\Star dynamical system. We sometimes denote by $\aB := \aA^G$ the corresponding fixed point C\Star algebra of the action $\alpha$ and we write $P_0:\aA \to \aA$ for the conditional expectation given by
\begin{equation*}
	P_0(x) := \int_G \alpha_g(x) \; dg.
\end{equation*}
At this point it is worth mentioning that all integrals over compact groups are understood to be with respect to probability Haar measure.
More generally, for an irreducible representation $(\pi,V)$ of $G$ we write $P_\pi:\aA \to \aA$ for the continuous $G$-equivariant projection onto the isotypic component $A(\pi):=P_\pi(\aA)$ given by
\begin{equation*}
	P_\pi(x) := \dim V \cdot \int_G \tr_V(\pi_g^*) \cdot \alpha_g(x) \;dg, 
\end{equation*}
where $\tr_V$ denotes the canonical trace on the algebra $\End(V)$ of linear endomorphisms of~$V$. It is a consequence of the Peter--Weyl Theorem \cite[Theorem 3.51 and Theorem 4.22]{HoMo06} that the algebraic direct sum $\bigoplus_{\pi\in\widehat{G}}A(\pi)$ is a dense $^*$-subalgebra of $\aA$. Here we write $\widehat{G}$ for the set of equivalence classes of all irreducible representations. Finally, we point out that each continuous group action $\sigma:P\times G\rightarrow P$ of a compact group $G$ on a compact space $P$ gives rise to a C\Star dynamical system $(C(P),G,\alpha_{\sigma})$ defined by 
\begin{align*}
\alpha_{\sigma}:G\times C(P)\rightarrow C(P), \quad (g,f)\mapsto f\circ\sigma_g.
\end{align*}

\subsubsection*{Hilbert Module Structures}

A huge part of this paper is concerned with Hilbert module structures. For the reader's convenience we recall some of the central definitions. Let $\alg B$ be a unital C\Star algebra. A~right pre-Hilbert $\aB$-module is a vector space $M$ which is a right $\aB$-module equipped with a positive definite $\aB$-valued sesquilinear form $\langle\cdot,\cdot\rangle_{\aB}$ satisfying
\begin{align*}
\langle x,y\cdot b\rangle_{\aB}=\langle x,y\rangle_{\aB}b \quad \text{and} \quad \langle x,y\rangle_{\aB}^*=\langle y,x\rangle_{\aB}
\end{align*}
for all $x,y\in M$ and $b\in\aB$. A~right Hilbert $\aB$-module is a right pre-Hilbert $\aB$-module~$M$ which is complete with respect to the norm given by $\norm{x}^2=\norm{\langle x,x\rangle_{\aB}}$ for $x\in M$. It is called a full right Hilbert $\aB$-module if the right ideal $J:=\Span\{\langle x,y\rangle_{\aB} \;|\; x,y\in M\}$ is dense in $\aB$. Since every dense ideal meets the invertible elements, in this case we have $J = \aB$. Left (pre-) Hilbert $\aB$-modules are defined in a similar way. Next, let $\aA$ and $\aB$ be unital C\Star algebras. A~right (pre-) Hilbert $\aA-\aB$-bimodule is a right (pre-) Hilbert $\aB$-module~$M$ that is also a left $\aA$-module satisfying
	\begin{align*}
		a\cdot(x\cdot b)=(a\cdot x)\cdot b \quad \text{and} \quad \langle a\cdot x,y\rangle_{\aB}=\langle x,a^*\cdot y\rangle_{\aB}
	\end{align*}
for all $x,y\in M$, $a\in\aA$ and $b\in\aB$.
We point out that right Hilbert $\aA-\aB$-bimodules are sometimes called $\aA-\aB$ correspondences in the literature. Given a right (pre-) Hilbert $\aA-\aB$-bimodule $M$ and a right (pre-) Hilbert $\aB-\aB$-bimodule~$M$, their algebraic $\aB$-tensor product $M\otimes_{\aB} N$ carries a natural right pre-Hilbert $\aA-\aB$-bimodule structure with right $\aB$-valued inner product given by
\begin{align}
		\langle x_1\otimes_{\aB} y_1,x_2\otimes_{\aB} y_2\rangle_{\aB}:=\langle y_1,\langle x_1,x_2\rangle_{\aB}\cdot y_2\rangle_{\aB} \label{tensor inner product}
\end{align}
for $x_1,x_2\in M$	and $y_1,y_2\in N$. In particular, its completion $M\widehat{\otimes}_{\aB} N$ with respect to the induced norm is a right Hilbert $\aA-\aB$-bimodule. Left (pre-) Hilbert $\aA-\aB$-bimodules are defined in a similar way. A~(pre-) Hilbert $\aA-\aB$-bimodule is an $\aA-\aB$-bimodule $M$ which is a left (pre-) Hilbert $\aA$-module and a right (pre-) Hilbert $\aB$-module satisfying
\begin{align*}
\langle a\cdot x,y\rangle_{\aB}=\langle x,a^*\cdot y\rangle_{\aB}, \quad {}_\aA\langle x\cdot b,y\rangle={}_\aA\langle x,y\cdot b^*\rangle \quad \text{and} \quad {}_\aA\langle x,y\rangle\cdot z=x\cdot\langle y,z\rangle_{\aB}
\end{align*}
for all $x,y,z\in M$, $a\in\aA$ and $b\in\aB$. A~Morita equivalence $\aA-\aB$-bimodule is a Hilbert $\aA-\aB$-bimodule with full inner products. The algebraic $\aB$-tensor product $M\otimes_{\aB} N$ of a (pre-) Hilbert $\aA-\aB$-bimodule $M$ and a (pre-) Hilbert $\aB-\aB$-bimodule $N$ carries a natural pre-Hilbert $\aA-\aB$-bimodule structure with inner products as in Equation (\ref{tensor inner product}).
		%
		%
Its completion $M\widehat{\otimes}_{\aB} N$ is a Hilbert $\aA-\aB$-bimodule. Finally, if $M$ is a Morita equivalence $\aA-\aB$-bimodule and $N$ a Morita equivalence $\aB-\aB$-bimodule, it is easily checked that the completion $M\widehat{\otimes}_{\aB} N$ is a Morita equivalence $\aA-\aB$-bimodule. For a detailed background on Hilbert module structures we refer to \cite{BeCoLe11,Bl98,EKQR06,La95,ReWi98}.

\section{On Free Group Actions: Some General Theory}
\label{sec:general_th}

The aim of this section is to discuss some of the forms of free actions of compact groups on C\Star algebras that have been used. In particular, we give some indications of their strengths and relationships to each other. Furthermore, we provide some methods to construct new free actions from given ones. As an example we present a one-parameter family of free $\SU(2)$-actions which is related to the Connes--Landi spheres. 

Throughout the following we use the symbol $\tensor_{\text{alg}}$ to denote the algebraic tensor product of vector spaces and we write $\tensor$ for the minimal tensor product of C\Star algebras.


\begin{prop}	\label{lemsatact}
	\emph{(\cite[\emph{Proposition 7.1.3}]{Ph87})}. Let $(\aA,G,\alpha)$ be a C\Star dynamical system with a unital C\Star algebra $\aA$ and a compact group $G$.
	Then the following definitions make a suitable completion of $\aA$ into a Hilbert $\aA\rtimes_{\alpha}G-{\aA}^G$-bimodule:
	\begin{itemize}
		\item[\emph{(i)}]  
			$f\cdot x:=\int_G f(g)\alpha_g(x)\;dg$ for $f\in L^1(G,\aA,\alpha)$ and $x\in \aA$.
		\item[\emph{(ii)}]  
			$x\cdot b:=xb$ for $x\in\aA$ and $b\in {\aA}^G$.
		\item[\emph{(iii)}] 
			${}_{\aA\rtimes_{\alpha}G}\langle x,y\rangle$ is the function $g\mapsto x\alpha_g(y^*)$ for $x,y\in \aA$.
		\item[\emph{(iv)}]  
			$\langle x,y\rangle_{{\aA}^G}:=\int_G\alpha_g(x^*y)\;dg$ for $x,y\in \aA$.
	\end{itemize}
\end{prop}

It is easily seen that the module under consideration in the previous statement is almost a Morita equivalence $\aA\rtimes_{\alpha}G-{\aA}^G$-bimodule. In fact, the only missing condition is that the range of ${}_{\aA\rtimes_{\alpha}G}\langle\cdot,\cdot\rangle$ need not be dense. The imminent definition was originally introduced by M. Rieffel and has a number of good properties that resemble the classical theory of free actions of compact groups as we will soon see below.

\begin{defn}\label{defsatact}(\cite[\emph{Definition 7.1.4}]{Ph87}).
Let $(\aA,G,\alpha)$ be a C\Star dynamical system with a unital C\Star algebra $\aA$ and a compact group $G$. We call $(\aA,G,\alpha)$ \emph{free} if the bimodule from Proposition \ref{lemsatact} is a Morita equivalence bimodule, that is, the range of ${}_{\aA\rtimes_{\alpha}G}\langle\cdot,\cdot\rangle$ is dense in the crossed product $\aA\rtimes_{\alpha}G$.
\end{defn}

\begin{rmk}
We point out that M. Rieffel used the notion ``\emph{saturated}'' instead of free, i.\,a., because of its relation to Fell bundles in the case of compact Abelian group actions. Moreover, we recall that \cite[Definition 1.6]{Ri88} provides a notion for free actions of locally compact groups which is consistent with Definition \ref{defsatact} for compact groups.
\end{rmk}

The next result shows that Definition \ref{defsatact} extends the classical notion of free actions of compact groups.

\begin{thm} \emph{(\cite[\emph{Proposition 7.1.12 and Theorem 7.2.6}]{Ph87})}. 
	\label{sat=free/comm}
Let $P$ be a compact space and $G$ a compact group. A~continuous group action $\sigma:P\times G\rightarrow P$ is free if and only if the corresponding C\Star dynamical system $(C(P),G,\alpha_{\sigma})$ is free in the sense of Definition~\ref{defsatact}.
\end{thm}

Another hint for the strength of Definition \ref{defsatact} comes from the following observation: Let $(P,X,G,q,\sigma)$ be a locally trivial principal bundle and $(\pi,V)$ a finite-dimensional unitary representation of $G$. Then it is a well-known fact that the isotypic component $C(P)(\pi)$ is finitely generated and projective as a right $C(X)$-module (\cf \cite[Proposition 8.5.2]{Wa11a}). In the C\Star algebraic setting a similar statement is valid.

\begin{thm}\label{isononcommvec}\emph{(\cite[\emph{Theorem 1.2}]{DKY12})}. 
Let $(\aA,G,\alpha)$ be a C\Star dynamical system with a unital C\Star algebra $\aA$ and a compact group $G$. Futhermore, let $(\pi,V)$ be a finite-dimensional unitary representation of $G$. If $(\aA,G,\alpha)$ is free, then the corresponding isotypic component $A(\pi)$ is finitely generated and projective as a right ${\aA}^G$-module.
\end{thm}

We proceed with introducing two more notions which will turn out to be equivalent characterizations of noncommutative freeness. The first notion is a C\Star algebraic version of the purely algebraic Hopf--Galois condition (\cf \cite{Haj04, Sch04}) and is due to D.\,A.~Ellwood.

\begin{defn}\label{ellwood}(\cite[\emph{Definition 2.4}]{Ellwood00}).
Let $(\aA,G,\alpha)$ be a C\Star dynamical system with a unital C\Star algebra $\aA$ and a compact group $G$. We say that $(\aA,G,\alpha)$ satisfies the \emph{Ellwood condition} if the map
\begin{align*}
\Phi:\aA\otimes_{\text{alg}}\aA\rightarrow C(G,\aA), \quad \Phi(x\otimes y)(g):=x\alpha_g(y)
\end{align*}
has dense range (with respect to the canonical C\Star norm on $C(G,\aA)$).
\end{defn}

The second notion is of representation-theoretic nature and makes use of the so-called generalized isotypic components of a C\Star dynamical system. 

\begin{defn}	\label{NCVB II}	
Let $(\aA,G,\alpha)$ be a C\Star dynamical system with a unital C\Star algebra $\aA$ and a compact group $G$. Furthermore, let $(\pi,V)$ be a finite-dimensional unitary representation of $G$. Then the space
\begin{align*}
				A_2(\pi):=\left\{s\in\aA \tensor \End(V) \;|\; \alpha_g(s) = s \cdot \pi_g \quad \text{for all} \quad g\in G\right\}
			\end{align*}
is called the \emph{generalized isotypic component} of $(\pi,V)$. It is easily checked that the canonical right action of the unital C\Star algebra
\begin{align*}
				\aC(\pi):=\{c\in\aA \tensor \End(V) \;|\;\alpha_g(c) = \pi^*_g \cdot c \cdot \pi_g\quad \text{for all} \quad g\in G\}
			\end{align*}
turns $A_2(\pi)$ into a right $\aC(\pi)$-module.
			
\end{defn}

The following statement describes the natural Hilbert bimodule structure of the generalized isotypic components. The arguments consist of straightforward computations and therefore we omit the proof.

\begin{prop}
Let $(\aA,G,\alpha)$ be a C\Star dynamical system with a unital C\Star algebra $\aA$ and a compact group $G$. Furthermore, let $(\pi,V)$ be a finite-dimensional unitary representation of $G$. Then the  following definitions make $A_2(\pi)$ into a $\aA^G\otimes\End(V)-\aC(\pi)$-Hilbert bimodule:
\begin{itemize}
	\item[\emph{(i)}] 
		$b.s:=bs$ for $b\in\aA^G\otimes\End(V)$ and $s\in A_2(\pi)$.
	\item[\emph{(ii)}]
		$s.c:=sc$ for $s\in A_2(\pi)$ and $c\in\aC(\pi)$.
	\item[\emph{(iii)}]
		${}_{\aA^G\otimes\End(V)}\langle s,t\rangle:=st^*$ for $s,t\in A_2(\pi)$.
	\item[\emph{(iv)}]
		$\langle s,t\rangle_{\aC(\pi)}:=s^*t$ for $s,t\in A_2(\pi)$.
\end{itemize}
\end{prop}

We are now ready to present the second notion which is of major relevance in our attempt to classify free C\Star dynamical systems.

\begin{defn}(\cite[Definition 1.1 (b)]{GoLaPe94}).
Let $(\aA,G,\alpha)$ be a C\Star dynamical system with a unital C\Star algebra $\aA$ and a compact group $G$. The \emph{Averson spectrum} of $(\aA,G,\alpha)$ is defined as
\begin{align*}
\text{Sp}(\alpha):=\Bigl\{[(\pi,V)]\in\widehat{G} \;\big{|}\; \Span\{\langle s,t\rangle_{\aC(\pi)}\;|\; s,t\in A_2(\pi)\}=\aC(\pi)\Bigr\}.
\end{align*}
That is, $[(\pi,V)]\in\text{Sp}(\alpha)$ if and only if the corresponding generalized isotypic component $A_2(\pi)$ is a full right $\aA^G\otimes\End(V)-\aC(\pi)$-Hilbert bimodule.
\end{defn}

As the following result finally shows, all the previous notions agree.

\pagebreak[3]
\begin{thm}\label{equcondsatact} 
Let $(\aA,G,\alpha)$ be a C\Star dynamical system with a unital C\Star algebra $\aA$ and a compact group $G$. Then the following statements are equivalent:
\begin{itemize}
	\item[\emph{(a)}] 
		The C\Star dynamical system $(\aA,G,\alpha)$ is free.
	\item[\emph{(b)}]
		The C\Star dynamical system $(\aA,G,\alpha)$ satisfies the Ellwood condition.
	\item[\emph{(c)}]
		The C\Star dynamical system $(\aA,G,\alpha)$ satisfies $\text{Sp}(\alpha)=\widehat{G}$.
\end{itemize}
\end{thm}

The equivalence between (a) and (b) was proved quite recently in \cite[Theorem 4.4]{DKY12}, although it has been known that these two conditions are closely related to each other (\cf~\cite{Wa10} for the case of compact Lie group actions). A~proof of the equivalence between (a) and (c) can be found in \cite[Lemma 3.1]{GoLaPe94}. 

We now focus our attention on compact Abelian groups. In fact, given a C\Star dynamical system $(\aA,G,\alpha)$ with a unital C\Star algebra $\aA$ and a compact Abelian group $G$, we first note that the definition of the isotypic component $A(\pi)$ corresponding to a character $\pi\in\widehat{G}=\Hom_{\text{gr}}(G,\mathbb{T})$ simplifies to
\begin{align*}
A(\pi)=\{a\in\aA\;|\; \alpha_g(a)=\pi_g\cdot a \quad \text{for all} \quad g\in G\}.
\end{align*}
Moreover, it is easily seen that $A_2(\pi)=A(\pi)$ and that $\aC(\pi)=\aA^G$. The next statement provides equivalent characterizations of a free action in the context of compact Abelian groups. It directly follows from Theorem \ref{equcondsatact} and the previous observations by repeatedly using the fact that $A(-\pi)=A(\pi)^*:=\{a^*\in\aA\;|\,a\in A(\pi)\}$ holds for each $\pi\in\widehat{G}$.

\begin{cor}\label{equcondsatactgrpcomm} 
Let $(\aA,G,\alpha)$ be a C\Star dynamical system with a unital C\Star algebra $\aA$ and a compact Abelian group $G$. Then the following conditions are equivalent:
\begin{itemize}
	\item[\emph{(a)}]
		The C\Star dynamical system $(\aA,G,\alpha)$ is free.
	\item[\emph{(b)}] 
		For each $\pi\in\widehat{G}$ the space $A(\pi)$ is a Morita equivalence $\aA^G-\aA^G$-bimodule. 
	\item[\emph{(c)}] 
		For each $\pi\in\widehat{G}$ the multiplication map on $\aA$ induces an isomorphism  between $A(-\pi)\widehat{\otimes}_{\aA^G}A(\pi)$ and $\aA^G$.
\end{itemize}
\end{cor}

As we will see soon, Corollary \ref{equcondsatactgrpcomm} gives rise to a first invariant for free actions of compact Abelian groups. For the time being, we continue with some examples to get more comfortable with free actions of compact Abelian groups.

\begin{expl}\label{expl sat act a}
Let $\theta$ be a real skew-symmetric $n\times n$ matrix. The \emph{noncommutative $n$-torus} $\mathbb{T}^n_{\theta}$ is the universal unital C\Star algebra generated by unitaries $U_1,\ldots,U_n$ with
$$U_rU_s=\exp(2\pi i\theta_{rs})U_sU_r \quad \text{for all} \quad 1\leq r,s\leq n.$$
It carries a continuous action $\alpha^n_{\theta}$ of the $n$-dimensional torus $\mathbb{T}^n$ by algebra automorphisms which is on generators defined by 
\begin{align*}
\alpha^n_{\theta,z}\bigr(U^{\mathbf k}\bigl):=z^{\mathbf k}\cdot U^{\mathbf k},
\end{align*}
where $z^{\mathbf k}:=z^{k_1}_1\cdots z^{k_n}_n$ and $U^{\mathbf k}:=U^{k_1}_1\cdots U^{k_n}_n$ for $z=(z_1,\ldots,z_n)\in\mathbb{T}^n$ and ${\mathbf k}:=(k_1,\ldots,k_n)\in\mathbb{Z}^n$. The isotypic component $(\mathbb{T}^n_{\theta})(\mathbf k)$ corresponding to the character ${\mathbf k}\in\mathbb{Z}^n$ is given by $\mathbb{C}\cdot U^{\mathbf k}$. In particular, each isotypic component contains invertible elements from which we conclude that the C\Star dynamical system $(\mathbb{T}^n_{\theta},\mathbb{T}^n,\alpha^n_{\theta})$ is free.
\end{expl}

\begin{expl}\label{expl sat act b}
	Let $H$ be the discrete, 3-dimensional Heisenberg group and let $C^*(H)$ denote its group C\Star algebra. Then $C^*(H)$ is the universal C\Star algebra 
	generated by unitaries $U$, $V$ and $W$ satisfying
	\[UW=WU, \quad VW=WV \quad \text{and} \quad UV=WVU.
	\]
	It carries a continuous action $\alpha$ of the 2-dimensional torus $\mathbb{T}^2$ by algebra automorphisms which is on generators defined by 
	\[
	\alpha_{(z,w)}(U^kV^lW^m):=z^kw^l\cdot U^kV^lW^m,
	\]where $(z,w)\in\mathbb{T}$ and $k,l,m\in\mathbb{Z}$. The corresponding fixed point algebra $\aB$ is the centre of $C^{*}(H)$ which is equal to the group $C^{*}$-algebra $C^{*}(Z)$ of the center $Z\cong\mathbb{Z}$ of $H$. Moreover, the isotypic component $C^{*}(H)_{(k,l)}$ corresponding to the character $(k,l)\in\mathbb{Z}^2$ is given by $\aB\cdot U^kV^l$. In particular, each isotypic component $C^{*}(H)_{(k,l)}$ contains invertible elements from which we conclude that the C\Star dynamical system $(C^{*}(H),\mathbb{T}^2,\alpha)$ is free. We point out that $C^{*}(H)$ serves as a ``universal'' noncommutative principal $\mathbb{T}^2$-bundle in \cite{ENOO09} and that its $K$-groups are isomorphic to $\mathbb{Z}^3$.
\end{expl}

\begin{expl}\label{expl sat act c}
For $q\in [-1,1]$ consider the C\Star algebra $\SU_q(2)$ from \cite{Wo87}.
We recall that it is the universal C\Star algebra generated by two elements $a$ and $c$ subject to the five relations
\begin{align*}
a^*a+cc^*=1, \quad aa^*+q^2cc^*=1, \quad cc^*=c^*c, \quad ac=qca \quad \text{and} \quad ac^*=qc^*a.
\end{align*}
It carries a continuous action $\alpha$ of the one-dimensional torus $\mathbb{T}$ by algebra automorphisms, which is on the generators defined by
\begin{align*}
\alpha_z(a):=z\cdot a \quad \text{and} \quad \alpha_z(c):=z\cdot c, \quad z\in\mathbb{T}.
\end{align*}
The fixed point algebra of this action is the \emph{quantum 2-sphere} $S^2_q$ and we call the corresponding C\Star dynamical system $(\SU_q(2),\mathbb{T},\alpha)$ the \emph{quantum Hopf fibration}. It is free according to \cite[Corollary 3]{Szy03}. In fact, the author shows that if $E$ is a locally finite graph with no sources and no sinks, then the natural gauge action on the graph C\Star algebra $C^*(E)$ is free. 
\end{expl}

\begin{rmk}\label{tncpb rmk}
	We recall that Example \ref{expl sat act a} and Example \ref{expl sat act b} are special cases of so-called \emph{trivial noncommutative principal bundles} as discussed in \cite{Wa11a, Wa11b,Wa11e}. In fact, it~is not hard to see that each trivial noncommutative principal bundle is free (\cf Remark~\ref{inv elts} and \cite{SchWa15}).
\end{rmk}

\begin{rmk}
Let $\mathbb{K}$ be the algebra of compact operators on some separable Hilbert space. The C\Star algebra $\SU_q(2)$ is described in \cite{Do80} as an extension of $C(\mathbb{T})$ by $C(\mathbb{T})\otimes\mathbb{K}$, \ie, by a short exact sequence 
\begin{align}
0\longrightarrow C(\mathbb{T})\otimes\mathbb{K}\longrightarrow\SU_q(2)\rightarrow C(\mathbb{T})\longrightarrow 0\label{su2 seq 1}
\end{align}
of C\Star algebras. If we consider $C(\mathbb{T})$ endowed with the canonical $\mathbb{T}$-action induced by right-translation, then a few moments thought shows that the sequence (\ref{su2 seq 1}) is in fact $\mathbb{T}$\ndash equivariant. In particular, it induces the following short exact sequence of C\Star algebras:
\begin{gather}
0\longrightarrow (C(\mathbb{T})\otimes\mathbb{K})\rtimes\mathbb{T}\longrightarrow\SU_q(2)\rtimes_{\alpha}\mathbb{T}\longrightarrow C(\mathbb{T})\rtimes\mathbb{T}\longrightarrow 0.\label{su2 seq 2}
\shortintertext{Since}
(C(\mathbb{T})\otimes\mathbb{K})\rtimes\mathbb{T}\cong(C(\mathbb{T})\rtimes\mathbb{T})\otimes\mathbb{K}\cong\mathbb{K}\otimes\mathbb{K}\cong\mathbb{K}\notag
\end{gather}
and $C(\mathbb{T})\rtimes\mathbb{T}\cong\mathbb{K}$ by the well-known Stone-von Neumann Theorem, we conclude from \cite[Proposition 6.12]{Ro04} that the crossed product $\SU_q(2)\rtimes_{\alpha}\mathbb{T}$ is stable. Moreover, the fact that the C\Star dynamical system $(\SU_q(2),\mathbb{T},\alpha)$ is free implies that the crossed product $\SU_q(2)\rtimes\mathbb{T}$ is Morita equivalent to the corresponding fixed-point algebra $S^2_q$ and thus they are also stably isomorphic by a famous result of Brown, Green and Rieffel (\cf \cite[Theorem 1.2]{BrGrRi77} or \cite[Section 5.5]{ReWi98}), \ie, $\SU_q(2)\rtimes\mathbb{T}\cong S^2_q\otimes\mathbb{K}$. The previous result affirms a question of the second author in the context of a notion of freeness which is related to \mbox{Green's} Theorem (\cf \cite[Corollary 15]{Green77} and \cite{ENOO09}). 
\end{rmk}

In the remaining part of this section we provide some methods to construct new free actions from given ones. As an application, we present a one-parameter family of free $\SU(2)$-actions which are related to the Connes--Landi spheres. 

\begin{prop}\label{restricted free}
Let $(\aA,G,\alpha)$ be a C\Star dynamical system with a unital C\Star algebra $\aA$ and a compact group $G$.
If $(\aA,G,\alpha)$ is free and $H$ a closed subgroup of $G$, then also the restricted C\Star dynamical system $(\aA,H,\alpha_{\mid H})$ is free.
\end{prop}
\begin{proof}
Since $(\aA,G,\alpha)$ satisfies the Ellwood condition, the surjectivity of the restriction map $C(G,\aA)\rightarrow C(H,\aA)$, $f\mapsto f_{\mid H}$ implies that the map
\begin{align*}
\Phi:\aA\otimes_{\text{alg}}\aA\rightarrow C(H,\aA), \quad \Phi(x\otimes y)(h):=x\alpha_h(y)
\end{align*}
has dense range. That is, $(\aA,H,\alpha_{\mid H})$ also satisfies the Ellwood condition.
\end{proof}

\begin{prop}\label{normal subgroup}
Let $(\aA,G,\alpha)$ be a C\Star dynamical system with a unital C\Star algebra $\aA$ and a compact group $G$.
If $(\aA,G,\alpha)$ is free and $N$ a closed normal subgroup of $G$, then also the induced C\Star dynamical system
$(\aA^N,G/N,\alpha_{\mid G/N})$ is free.
\end{prop}
\begin{proof}
Since $(\aA,G,\alpha)$ satisfies the Ellwood condition, the map
\begin{align*}
\Phi:\aA\otimes_{\text{alg}}\aA\rightarrow C(G,\aA), \quad \Phi(x\otimes y)(g):=x\alpha_g(y)
\end{align*}
has dense range. Moreover, the C\Star algebra $C(G/N, \aA^N)$ is naturally identified with functions in $C(G,\aA)$ satisfying $f(g) = \alpha_{n1}( f(gn_2))$ for all $g\in G$ and $n_1,n_2\in N$. In~other words, $C(G/N, \aA^N)$ is the fixed point algebra of the action $\alpha\tensor\text{rt}$ of $N\times N$ on $C(G)\otimes\aA=C(G,\aA)$, where $\text{rt}:N\times C(G)\rightarrow C(G)$, $\text{rt}(n,f)(g):=f(gn)$ denotes the right-translation action by $N$. Let $P_N:\aA\rightarrow\aA$ and $P_{N \times N}:C(G,\aA) \rightarrow C(G,\aA)$ be the conditional expectations for the actions $\alpha_{\mid N}$ and $\alpha\tensor\text{rt}$, respectively. Then we obtain for arbitrary $x,y\in\aA$
\begin{align*}
\Phi( P_N(x) \tensor P_N(y) )
	&= \int_{N \times N} \alpha_{n_1}(x) \alpha_{g n_2}(y) \;dn_1dn_2 = \int_{N \times N} \alpha_{n_1}(x) \alpha_{n_2 g}(y) \;dn_1dn_2
	\\
	&= \int_{N \times N} \alpha_{n_1}( x \alpha_{n_2 g}(y) ) \;dn_1dn_2 = \int_{N \times N} \alpha_{n_1}( x \alpha_{g n_2}(y) ) \;dn_1dn_2
	\\
	&=P_{N \times N}( \Phi(x \tensor y)).
\end{align*}
It follows that the restricted map
\begin{align*}
\Phi_{\mid \aA^N\otimes_{\text{alg}}\aA^N}:\aA^N\otimes_{\text{alg}}\aA^N\rightarrow C(G,A), \quad \Phi(x\otimes y)(g):=x\alpha_g(y)
\end{align*}
has dense range in the C\Star subalgebra $C(G/N,\aA^N)$. That is, $(\aA^N,G/N,\alpha_{\mid G/N})$ also satisfies the Ellwood condition. 
\end{proof}

\begin{prop}	\label{newsatact B}
	Let $(\aA, G, \alpha)$ and $(\aC, H, \gamma)$ be C\Star dynamical systems with unital C\Star algebras $\aA$, $\aC$ and compact groups $G$, $H$. If $(\aA, G, \alpha)$ and $(\aC, H, \gamma)$ are free, then also their tensor product $(\aA \tensor \aC, G \times H, \alpha \tensor \gamma)$ is free.
\end{prop}
\begin{proof}
We first note that the map
\begin{gather*}
	\Phi:\aA\otimes_{\text{alg}}\aC\otimes_{\text{alg}}\aA\otimes_{\text{alg}}\aC\rightarrow C(G\times H,\aA\otimes \aC), 
	\\ 
	\Phi(x\otimes y\otimes u\otimes v)(h) := x\alpha_h(y)\otimes u\gamma_h(v)
\end{gather*}
is, up to a permutation of the tensor factors, an amplification of the corresponding maps induced by $(\aA,G,\alpha)$ and $(\aC,H,\gamma)$. Therefore, $(\aA\otimes\aC,G\times H,\alpha\otimes\gamma)$ inherits the Ellwood condition from $(\aA,G,\alpha)$ and $(\aC,H,\gamma)$.
\end{proof}

\begin{rmk}\label{another new free action}
	Suppose that $(\aA,G,\alpha)$ is a free C\Star dynamical system with a unital C\Star al\-ge\-bra~$\aA$ and a compact group $G$. Furthermore, let $\aC$ be an arbitrary unital C\Star algebra. Then Proposition \ref{newsatact B} applied to the trivial group $H$ implies that the C\Star dynamical system $(\aA\otimes\aC,G,\alpha\otimes\id_{\aC})$ is free. More generally, if $(\aC,G,\gamma)$ is an arbitrary C\Star dynamical system, it is not hard to check that the C\Star dynamical system $(\aA\otimes\aC,G,\alpha\otimes\gamma)$ satisfies the Ellwood condition, \ie, $(\aA\otimes\aC,G,\alpha\otimes\gamma)$ is free. This observation corresponds in the classical setting to the situation of endowing the cartesian product of a free and compact $G$-space $X$ and any compact $G$-space $Y$ with the free diagonal action of $G$.
\end{rmk}


\begin{thm}\label{newsatact}
	Let $(\aA,G,\alpha)$ and $(\aC,H,\gamma)$ be free C\Star dynamical systems with unital C\Star algebras $\aA$, $\aC$ and compact groups $G$, $H$. Furthermore, let $(\aA,H,\beta)$ be any another C\Star dynamical system such that the actions $\alpha$ and $\beta$ commute. Then the following assertions hold:
	\begin{itemize}
	\item[\emph{(}a\emph{)}]
	The C\Star dynamical system $\bigr((\aA\otimes\aC),G\times H,(\alpha\circ\beta)\otimes\gamma\bigl)$ is free.
	\item[\emph{(}b\emph{)}]
		The C\Star dynamical system $\bigr((\aA\otimes\aC)^H,G,\alpha\otimes\id_{\aC}\bigl)$ is free, where the fixed space $(\aA \tensor \aC)^H$ is taken with respect to the tensor product action $\beta \tensor \gamma$ of~$H$.
	\end{itemize}
\end{thm}
\begin{proof}
(a) We first note that $(\aA,G,\alpha)$ and $(\aC,H,\gamma)$ both satisfy the Ellwood condition from which we conclude that the maps
\begin{align*}
\Phi_1:\aA\otimes_{\text{alg}}\aA\otimes_{\text{alg}}\aC\rightarrow C(G,\aA\otimes\aC), \quad \Phi_1(x_1\otimes x_2\otimes y)(g):=x_1\alpha_g(x_2)\otimes y
\end{align*}
and $\Phi_2:(\aA\otimes_{\text{alg}}\aC)\otimes_{\text{alg}}(\aA\otimes_{\text{alg}}\aC)\rightarrow C(H,\aA\otimes\aA\otimes\aC)$ given by
\begin{align*}
\Phi_2\bigr((x_1\otimes y_1)\otimes(x_2\otimes y_2)\bigl)(h):=x_1\otimes\beta_h(x_2)\otimes y_1\gamma_h(y_2)
\end{align*}
have dense range. It follows, identifying $C(H,C(G,\aA\otimes\aC))$ with $C(G\times H,\aA\otimes\aC)$, that also their amplified composition
\begin{gather*}
\Phi:=(\id_{C(H)}\otimes\Phi_1)\circ\Phi_2:(\aA\otimes_{\text{alg}}\aC)\otimes_{\text{alg}}(\aA\otimes_{\text{alg}}\aC)\rightarrow C(G\times H,\aA\otimes\aC)
\shortintertext{given by}
\Phi\bigr((x_1\otimes y_1)\otimes(x_2\otimes y_2)\bigl)(g,h)=x_1(\alpha_g\beta_h)(x_2)\otimes y_1\gamma_h(y_2)
\end{gather*}
has dense range. That is, $\bigr((\aA\otimes\aC),G\times H,(\alpha\circ\beta)\otimes\gamma\bigl)$ satisfies the Ellwood condition. 

(b) To verify the second assertion we simply apply Proposition \ref{normal subgroup} to the C\Star dynamical system in part (a) and the normal subgroup $\{\one_G\}\times H$ of $G\times H$.
\end{proof}

\begin{expl}
	\sloppy
	The Connes-Landi spheres $\mathbb{S}^n_{\theta}$ are extensions of the noncommutative tori $\mathbb{T}^n_{\theta}$ (\cf \cite{LaSu05}).
	We are in particularly interested in the case $n=7$. In this case there is a continuous action of the 2-torus $\mathbb{T}^2$ on the 7-sphere $\mathbb{S}^7\subseteq\mathbb{C}^4$ given by
	\begin{align*}
	\sigma:\mathbb{S}^7\times\mathbb{T}^2\rightarrow\mathbb{S}^7, \quad \bigr((z_1,z_2,z_3,z_4),(t_1,t_2)\bigl)\mapsto(t_1z_1,t_1z_2,t_2z_3,t_2z_4).
	\end{align*}
	Let $(C(\mathbb{S}^7),\mathbb{T}^2,\alpha_{\sigma})$ be the corresponding C\Star dynamical system. Furthermore, let $(\mathbb{T}^2_{\theta},\mathbb{T}^2,\alpha^2_{\theta})$ be the free C\Star dynamical system associated to the gauge action on the noncommutative 2-torus $\mathbb{T}^2_{\theta}$ (see Example \ref{expl sat act a} below). The Connes--Landi sphere $\mathbb{S}^7_{\theta}$ is defined as the fixed point algebra of the tensor product action $\alpha_{\sigma}\otimes\alpha^2_{\theta}$ of $\mathbb{T}^2$ on $C(\mathbb{S}^7,\mathbb{T}^2_{\theta})=C(\mathbb{S}^7)\otimes\mathbb{T}^2_{\theta}$, \ie,
	\begin{align*}
	\mathbb{S}^7_{\theta}:=C(\mathbb{S}^7,\mathbb{T}^2_{\theta})^{\mathbb{T}^2}.
	\end{align*}
	Our intention is to use Theorem \ref{newsatact} (b) to endow $\mathbb{S}^7_{\theta}$ with a free $\SU(2)$-action. For this purpose, we consider the free and continuous $\SU(2)$-action on the 7-sphere $\mathbb{S}^7$ given by
	\begin{align*}
	\mu:\mathbb{S}^7\times \SU(2) \rightarrow\mathbb{S}^7, \quad \bigr((z_1,z_2,z_3,z_4),M\bigl)\mapsto(z_1,z_2,z_3,z_4)\begin{pmatrix}
	 M &   0\\
	 0 & M\\
	\end{pmatrix}.
	\end{align*}
	It follows from Theorem \ref{sat=free/comm} that the induced C\Star dynamical system $(C(\mathbb{S}^7),\SU(2),\alpha_{\mu})$ is free. Moreover, it is easily verified that the actions $\alpha_{\mu}$ and $\alpha_{\sigma}$ commute. Therefore, Theorem \ref{newsatact} (b) implies that the C\Star dynamical system $\bigr(\mathbb{S}^7_{\theta},\SU(2),\alpha_{\mu}\otimes\id_{\mathbb{T}^2_{\theta}}\bigl)$ is free.
\end{expl}

\section{Construction of Free Actions of Compact Abelian Groups}
\label{prelud}

In this section we will specify the data on which our classification of free actions of compact Abelian groups is based (see Definition~\ref{def:factor_sys}) and, moreover, we will show that it is complete in the sense that every classifying data indeed can be obtained from a free C\Star dynamical system.

\subsection{The Picard Group and Factor Systems}

Corollary~\ref{equcondsatactgrpcomm} suggests the relevance of Morita equivalence $\aB-\aB$ bimodules for the classification of free actions with a given fixed point algebra $\aB$ and a given compact Abelian group~$G$. These objects have a natural interpretation as noncommutative line bundles ``over'' $\aB$. In particular, just like in the classical theory of line bundles, the set of their equivalence classes carry a natural group structure.

\begin{defn}	\label{picgrpalg}
Let $\aB$ be a C\Star algebra. Then the set of equivalence classes of Morita equivalence $\aB-\aB$-bimodules forms an Abelian group with respect to the internal tensor product of Hilbert $\aB-\aB$-bimodules. This group is called the \emph{Picard group} of $\aB$ and it is denoted by $\Pic(\aB)$.
\end{defn}

\begin{rmk}	\label{out_in_pic}
	For a unital C\Star algebra $\aB$, the group of outer automorphisms $\Out(\aB) := \Aut(\aB) / \Inn(\aB)$ is always a subgroup of $\Pic(\aB)$. More precisely, for any \Star automorphism~$\alpha$ of $\aB$ we may define an element of the Picard group $\Pic(\aB)$ by the following Morita equivalence $\aB-\aB$-bimodule. Let $M_\alpha$ be the vector space $\aB$ endowed with the canonical left Hilbert $\aB$-module structure. Moreover, let the right action be given by $m \cdot b := m \alpha(b)$ for $m \in M_\alpha$ and $b \in \aB$, and let the right $\aB$-valued inner product be given by $\langle m_1, m_2 \rangle_\aB := \alpha^{-1}(m_1^* m_2)$ for $m_1, m_2 \in M_\alpha$. It is straightforwardly checked that $M_\alpha$ is a Morita equivalence $\aB-\aB$-bimodule and that, for $\alpha, \beta \in \Aut(\aB)$, we have $M_\alpha \widehat{\otimes}_{\aB} M_\beta \simeq M_{\alpha \circ \beta}$. A~few moments thought also shows that $M_\alpha \simeq \aB$ iff $\alpha$ is inner. Summarizing we have the exact sequence of groups
	\begin{align*}
		1\longrightarrow\Inn(\aB){\longrightarrow}\Aut(\aB){\longrightarrow}\Pic(\aB).
	\end{align*}
\end{rmk}

\begin{expl}~\label{expl:pic_grp}
	\begin{enumerate}
	\item	
		For a finite-dimensional C\Star algebra $\aB$, the Picard group $\Pic(\aB)$ is isomorphic to the group of permutations of the spectrum of $\aB$ (see \cite[Section 3]{BrGrRi77}).
	\item	\label{en:pic_grp:C(X)}
		For some compact space $X$, the Picard group $\Pic \bigl( C(X) \bigr)$ is isomorphic to the semi-direct product $\Pic(X)\rtimes\Homeo(X)$, where $\Pic(X)$ denotes the set of equivalence classes of complex line bundles over $X$ and $\Homeo(X)$ the group of homeomorphisms of $X$ (see \cite{BrGrRi77,BuWe04}).
	\item
		Let $0 < \theta < 1$ be irrational and $\mathbb{T}^2_{\theta}$ the corresponding quantum 2-torus. Then $\Pic(\mathbb{T}^2_{\theta})$ is isomorphic to $\Out(\mathbb{T}^2_{\theta})$ in case $\theta$ is quadratic and isomorphic to $\Out(\mathbb{T}^2_{\theta})\rtimes\mathbb{Z}$ otherwise (see~\cite{Ko97}).
	\end{enumerate}
\end{expl}

The next statement is a first step towards finding invariants, \ie, classification data, for free C\Star dynamical systems with a prescribed fixed point algebra.

\begin{prop}\label{prop:sat_vs_pic}
	Each free C\Star dynamical system $(\aA,G,\alpha)$ with unital C\Star algebra $\aA$, compact Abelian group $G$ and fixed point algebra $\aB:=\aA^G$ gives rise to a group homomorphism $\varphi_\aA:\widehat{G}\rightarrow\Pic(\aB)$ given by $\varphi_\aA(\pi):=[A(\pi)]$.
\end{prop}
\begin{proof}
	To verify the assertion we choose $\pi,\rho\in\widehat{G}$ and use Corollary \ref{equcondsatactgrpcomm} to compute
	\begin{align*}
	\varphi_\aA(\pi+\rho)=[A(\pi+\rho)]=[A(\pi)\widehat{\otimes}_{\aB}A(\rho)]=[A(\pi)][A(\rho)]=\varphi_\aA(\pi)\varphi_\aA(\rho).
	\end{align*}
	This shows that the map $\varphi_\aA$ is a group homomorphism.
\end{proof}

The group homomorphism $\varphi: \widehat{G} \to \Pic(\aB)$ in Proposition~\ref{prop:sat_vs_pic} is not enough to uniquely determine the C\Star dynamical system up to equivalence. Loosely speaking it determines the linear structure but not the multiplication of~$\aA$. In order to see what is missing, we choose for each $\pi \in \widehat{G}$ a Morita equivalence bimodule $M_\pi$ in the class $\varphi(\pi)$. The group homomorphism property of $\varphi$ guarantees that for each $\pi, \rho \in \widehat{G}$ there exist isomorphisms of Morita equivalence $\aB-\aB$-bimodules
\begin{equation*}
	\Psi_{\pi, \rho}: M_\pi \widehat{\otimes}_{\aB} M_\rho \to M_{\pi + \rho}.
\end{equation*}
In general $\varphi$ does not impose any relation among the maps $\Psi_{\pi, \rho}$.
But, comming from a free C\Star dynamical system $(\aA, G, \alpha)$ with unital C\Star algebra $\aA$, compact Abelian group $G$ and fixed point algebra $\aB:=\aA^G$, we may choose canonically 
\begin{align*}
M_\pi := A(\pi) \quad \text{and} \quad \Psi_{\pi, \rho}(x \tensor_\aB y) := xy
\end{align*}
for each $\pi, \rho \in \widehat{G}$. In this case, the associativity of the multiplication in $\aA$ implies 
\begin{equation}
	\label{eq:assoc}
	\Psi_{\pi + \rho, \sigma} \circ (\Psi_{\pi,\rho} \tensor_\aB \id_\sigma) 
	=
	\Psi_{\pi, \rho + \sigma} \circ (\id_\pi \tensor_\aB \Psi_{\rho, \sigma}).
\end{equation}
This suggests to take the following notion of factor system as a classifying object.

\begin{defn} \label{def:factor_sys}
	Let $\aB$ be a unital C\Star algebra, $G$ be a compact Abelian group, and let $\varphi:\widehat{G} \to \Pic(\aB)$ be a group homomorphism. Furthermore, let $(M_\pi, \Psi_{\pi, \rho})_{\pi, \rho \in \widehat{G}}$ be a family where
	\begin{enumerate}
	\item 
		for each $\pi \in \widehat{G}$ we have a Morita equivalence $\aB-\aB$-bimodule $M_\pi$  in the class~$\varphi(\pi)$.
	\item
		for each $\pi, \rho \in \widehat{G}$ we have an isomorphism $\Psi_{\pi,\rho}: M_\pi \widehat{\otimes}_{\aB} \M_\rho \to M_{\pi + \rho}$ of Morita equivalences.
	\end{enumerate}
	Then $(M_\pi, \Psi_{\pi, \rho})_{\pi, \rho \in \widehat{G}}$ is called a \emph{factor system} (for the map $\varphi$) if it satisfies equation~\eqref{eq:assoc} for all $\pi, \rho \in \widehat{G}$ and the normalization condition $(M_0, \Psi_{0,0}) = (\aB, \id_{\aB})$.
\end{defn}

\pagebreak[3]
\begin{rmk}~\label{rmk nice eq}
	\begin{enumerate}
	\item
		Up to the canonical isomorphism $M_\pi \widehat{\otimes}_{\aB} \aB \simeq M_\pi \simeq \aB \widehat{\otimes}_{\aB} M_\pi$, the normalization condition implies for all $\pi \in \widehat{G}$
		\begin{align*}
			\Psi_{\pi,0} &= \id_\pi = \Psi_{0,\pi} 
			&
			&\text{and}
			&
			\Psi_{\pi, -\pi} \tensor_{\aB} \id_\pi 
			&= \id_\pi \tensor_{\aB} \Psi_{-\pi, \pi}.  
		\end{align*}
	\item
		The group homomorphism $\varphi:\widehat{G} \to \Pic(\aB)$ can obviously be recovered from the factor system.
	\item
		Each free C\Star dynamical system $(\aA,G,\alpha)$ with unital C\Star algebra $\aA$ and compact Abelian group $G$ gives rise to a factor system $(A(\pi),m_{\pi, \rho})_{\pi, \rho\in\widehat{G}}$ for the group homomorphism $\varphi_\aA$ (\cf Proposition \ref{prop:sat_vs_pic}) with
	\begin{align*}
	m_{\pi, \rho}:A(\pi)\widehat{\otimes}_{\aB}A(\rho)\rightarrow A(\pi+\rho), \quad x\otimes_{\aB}y\mapsto xy.
	\end{align*}		
	\item
	     Given a group homomorphism $\varphi:\widehat{G} \to \Pic(\aB)$, the existence of a factor system imposes a non-trivial cohomological condition on $\varphi$. In Section 5 we will characterize this condition and provide cohomological ways to construct factor systems without a C\Star dynamical system at hand.
	\end{enumerate}
\end{rmk}

For the rest of this section we will show how to construct a corresponding C\Star dynamical system from a given factor system. This is done by reverse engineering an adaption of the GNS-representation for C\Star dynamical systems with arbitrary fixed point algebra. For this reason, we briefly review the construction. Readers familiar with the material may skip the next subsection.

\subsection{Adapting the GNS-Construction}

Let $(\aA, G, \alpha)$ be a C\Star dynamical system and $P_0:\aA \to \aA$ the conditional expectation onto the fixed point algebra $\aB: = \aA^G$. Then $\aA$ can be equipped with the structure of a right pre-Hilbert $\aB-\aB$-bimodule with respect to the usual multiplication and the inner product given by $\scal{x,y}_\aB := P_0(x^*y)$ for $x,y \in \aA$. Since $P_0$ is faithful, this inner product on $\aA$ is definite and we may take the completion of $\aA$ with respect to the norm $\norm{x}_2 := \norm{P_0(x^* x)}^{1/2}$. This provides a right Hilbert $\aB-\aB$-bimodule $L^2(\aA)$ with $\aA$ as a dense subset (\cf Proposition~\ref{lemsatact}). For each $\pi \in \widehat{G}$ the projection $P_\pi$ onto the isotypic component $A(\pi)$ can be continuously extended to a self-adjoint projection on $L^2(\aA)$. In~particular, the sets $A(\pi)$ are closed, pairwise orthogonal right Hilbert $\aB-\aB$-subbimodules of $L^2(\aA)$ and $L^2(\aA)$ can be decomposed into
\begin{equation*}
	L^2(\aA) = \overline{\bigoplus_{\pi \in \widehat{G}} A(\pi)}^{\norm{\cdot}_2}	
\end{equation*}
as right Hilbert $\aB-\aB$-bimodules. For each element $a \in \aA$ the left multiplication operator $\lambda_a:\aA\rightarrow \aA$, $x\mapsto ax$, then extends to an adjointable operator on $L^2(\aA)$. The arising representation
\begin{equation*}
	\lambda:\aA \to \End \bigl( L^2(\aA) \bigr), \quad a \mapsto \lambda_a
\end{equation*}
is called the \emph{left regular representation} of $\aA$. For each $g \in G$ the automorphism $\alpha_g$ extends from $\aA$ to an automorphism $U_g:L^2(\aA) \to L^2(\aA)$ of right Hilbert $\aB-\aB$-bimodules and the strongly continuous group $(U_g)_{g\in G}$ implements $\alpha_g$ in the sense that 
\begin{equation*}
	\alpha_g \bigl( \lambda_a \bigr) = U_g \, \lambda_a \, U_g^+	
\end{equation*}
for all $a \in \aA$. The vector $\one_\aB = \one_\aA \in L^2(\aA)$ is obviously cyclic and separating for this representation. In particular, the left regular representation is faithful and we may identify $\aA$ with the subalgebra $\lambda(\aA)$. 
Since the sum of the isotypic components is dense in $\aA$, the C\Star algebra $\lambda(\aA)$ is in fact generated by the operators $\lambda_a$ with $a \in A(\pi)$, $\pi \in \widehat{G}$. For such elements $a$ in some fixed isotypic component $A(\pi)$, $\pi \in \widehat{G}$, the operator $\lambda_a$ maps each subset $A(\rho) \subseteq L^2(\aA)$, $\rho \in \widehat{G}$, into $A(\pi+\rho)$ and therefore it is determined by the multiplication map
\begin{equation*}
	m_{\pi, \rho}: A(\pi) \tensor_{\text{alg}} A(\rho) \to A(\pi + \rho),
	\quad
	m_{\pi, \rho}(x,y) := xy = \lambda_x(y).
\end{equation*}
It is easily verified that $m_{\pi, \rho}$ factors to an isometry of the right Hilbert $\aB-\aB$-bimodules $A(\pi) \widehat{\tensor}_\aB A(\rho)$ and $A(\pi + \rho)$.

\subsection{Associativity and Factor Systems}
\label{sec:associativity cmpt ab}

In what follows we consider a fixed unital C\Star algebra $\aB$, a compact Abelian group $G$, and a group homomorphism $\varphi:\widehat{G}\rightarrow\Pic(\aB)$. We choose for each $\pi\in\widehat{G}$ a Morita equivalence $\aB-\aB$-bimodule $M_{\pi}$ in the isomorphism class $\varphi(\pi)\in\Pic(\aB)$, where for $\pi=0$ we choose $M_0 := \aB$. Since $\varphi$ is a group homomorphism, the Morita equivalence $\aB-\aB$-bimodules $M_{\pi}\widehat{\otimes}_{\aB}M_{{\rho}}$ and $M_{\pi+\rho}$ must be isomorphic for all $\pi,\rho\in\widehat{G}$ with respect to some isomorphism 
\begin{align*}
\Psi_{\pi,\rho}:M_{\pi}\widehat{\otimes}_{\aB}M_{\rho}\rightarrow M_{\pi+\rho}
\end{align*}
of Morita equivalence $\aB-\aB$-bimodules. We fix a choice of $\Psi_{\pi,\rho}$ for each $\pi, \rho \in \widehat{G}$ where $\Psi_{0,0} := \id_\aB$. This provides a bilinearmap 
\begin{align*}
m_{\pi,\rho}:M_{\pi}\times M_{\rho}\rightarrow M_{\pi+\rho}, \quad m_{\pi,\rho}(x,y):=\Psi_{\pi,\rho}(x\otimes_{\aB}y).
\end{align*}
The family of all such maps $(m_{\pi,\rho})_{\pi,\rho\in\widehat{G}}$ now gives rise to a multiplication map $m$ on the algebraic vector space 
\begin{align*}
A=\bigoplus_{\pi\in\widehat{G}}M_{\pi}.
\end{align*}

\begin{prop}
	\label{thm:factor_sys cptab}
	The following statements are equivalent:
	\begin{equivalence}
	\item[\emph{(a)}]
		$m$ is associative, \ie, $A$ is an algebra.
	\item[\emph{(b)}]
		$(M_\pi, \Psi_{\pi, \rho})_{\pi, \rho \in \widehat{G}}$ is a factor system.
	\end{equivalence}
\end{prop}
\begin{proof}
	For given $\pi, \rho, \sigma \in \widehat{G}$ we explicitly compute
	for all $x \in \M_\pi$, $y \in M_\rho$, $z \in M_\sigma$:
	\begin{align*}
		m \bigl( x, \, m(y,z) \bigr)
		&= m \bigl( x, \Psi_{\rho, \sigma}(y\otimes_{\aB}z) \bigr)
		= \Psi_{\pi, \rho +\sigma} \bigl( x, \Psi_{\rho, \sigma}(y\otimes_{\aB}z)\bigr) 
		\\
		m \bigl( m(x,y), \, z \bigr)
		&= m \bigl( \Psi_{\pi,\rho}(x\otimes_{\aB}y) , \, z \bigr)
		= \Psi_{\pi +\rho, \sigma}\bigl(\Psi_{\pi,\rho}(x\otimes_{\aB}y) , \, z \bigr).
	\end{align*}
	Therefore, $m$ is associative if and only if equation \eqref{eq:assoc} holds for all $\pi, \rho, \sigma \in \widehat{G}$.
\end{proof}

\subsection{Construction of an Involution}

We continue with a fixed factor system $(M_\pi, \Psi_{\pi, \rho})_{\pi, \rho\in\widehat{G}}$ for the map $\varphi$ and we write $A$ for the associated algebra. Our goal is to turn $A$ into a \Star algebra and right pre-Hilbert $\aB-\aB$-bimodule. For this purpose we involve the right Hilbert $\aB-\aB$-bimodule structure of each direct summands $M_\pi$ of $A$, \ie, the right $\aB$-valued inner products~$\langle\cdot,\cdot\rangle_\pi$ on~$M_\pi$.

\begin{lemma}\label{inv I}
The map $\langle\cdot,\cdot\rangle:A\times A\rightarrow\aB$ defined for $x = \bigoplus_\pi x_\pi, y = \bigoplus_\pi y_\pi \in A$ by
\begin{align*}
\langle x,y\rangle_{\aB}:=\sum_{\pi\in\widehat{G}}\langle x_\pi,y_\pi\rangle_\pi
\end{align*}
turns $A$ into a right pre-Hilbert $\aB-\aB$-bimodule and satisfies
\begin{align}
\langle m(b,x),m(b,x)\rangle_{\aB}\leq\norm{b}^2\langle x,x\rangle_{\aB}\label{cool inequality A}
\end{align}
for all $x\in A$ and $b\in\aB$.
\end{lemma}
\begin{proof}
The necessary computations are straightforward and thus left to the reader. We only point out that the inequality (\ref{cool inequality A}) is a consequence of the corresponding inequalities satisfied by the right $\aB$-valued inner products $\langle\cdot,\cdot\rangle_\pi$.
\end{proof}

\begin{lemma}\label{inv II}
For each $y\in M_{\pi}$ and every $\rho \in \widehat{G}$ the left multiplication operator 
\begin{align*}
\ell_y:M_{\rho}\rightarrow M_{\pi}\widehat{\otimes}_{\aB} M_\rho, \quad x\mapsto y\otimes_{\aB} x
\end{align*}
is adjointable and hence bounded with adjoint given by
\begin{align*}
	\ell_y^+:M_{\pi}\widehat{\otimes}_{\aB} M_\rho\rightarrow M_{\rho}, \quad z\otimes_{\aB} x\mapsto \langle y,z\rangle_{\aB}. x.
\end{align*}
\end{lemma}
\begin{proof}
To verify the assertion we first note that the linear span of elements $z\otimes_{\aB}z'$ with $z\in M_\pi$ and $z' \in M_\rho$ is dense in $M_{\pi+\rho}$. For such an element and $x\in M_\rho$ we obtain
	\begin{align*}
		\scal{\ell_y(x), \; z\otimes_{\aB}z'}_{\aB}
		= \scal{y\otimes_{\aB}x, z\otimes_{\aB}z'}_{\aB}
		= \scal[\big]{x,\langle y,z\rangle\cdot z'}_{\aB}=\scal[\big]{x,\ell_y^+(z\otimes_{\aB}z')}_{\aB}
	\end{align*}
which implies that $\ell_y$ is adjointable with adjoint given by the map $\ell_y^+$.
\end{proof}

\begin{prop}\label{inv XI}
	For each $y\in M_{\pi}$ and every $\rho \in \widehat{G}$ the left multiplication operator 
	\begin{align*}
	\lambda_y:M_\rho\rightarrow M_{\pi+\rho}\,\,\,\lambda_y(x) := m(y,x)
	\end{align*}
	is adjointable and hence bounded and satisfies
	\begin{align}
		\scal{\lambda_y(x), \; \lambda_y(x)}_{\aB}
		\leq \norm{\langle y,y \rangle_{\aB}} \cdot \langle x,x\rangle_{\aB}
		\label{cool inequality C}
	\end{align}
	for all $x\in M_\rho$.
\end{prop}
\begin{proof}
That the left multiplication operator $\lambda_y:M_\rho\to M_{\pi+\rho}$ is adjointable for each $y\in M_{\pi}$ and every $\rho \in \widehat{G}$ is an immediate consequence of Lemma \ref{inv II} and the unitarity of the map $\Psi_{\pi,\rho}$ because $\lambda_y=\Psi_{\pi,\rho}\circ\ell_y$. The asserted inequality (\ref{cool inequality C}) then easily follows from a short computation involving inequality (\ref{cool inequality A}). Indeed, we obtain
\begin{align*}
\scal{\lambda_y(x), \; \lambda_y(x)}_{\aB}&=\scal{\Psi_{\pi,\rho}(y\otimes_{\aB}x), \Psi_{\pi,\rho}(y\otimes_{\aB}x)}_{\aB}=\scal{y\otimes_{\aB}x, y\otimes_{\aB}x}_{\aB}\\
&=\scal[\big]{x,\langle y,y\rangle_{\aB}\cdot x}_{\aB}
\leq\norm{\langle y,y \rangle_{\aB}}\langle x,x\rangle_{\aB}
\end{align*}
for all $x\in M_\rho$.
\end{proof}

\begin{cor}\label{inv III}
For each $a\in A$ the left multiplication operator
	\begin{equation*}
		\lambda_a: A\rightarrow A, \quad \lambda_a(x) := m(a,x)
	\end{equation*}
is adjointable and bounded.
\end{cor}

We are now ready to introduce an involution on $A$ which turns $A$ into a \Star algebra. Here we use the fact that the involution is determined by the inner product if we impose that the inner product on $A$ takes it canonical form.

\begin{defn}\label{inv IV}
The adjoint map $i:A\rightarrow A$, $a \mapsto i(a)$ is given by
\begin{align*}
i(a):=\lambda^+_a(\one_{\aB}).
\end{align*}
It is clearly antilinear and maps the subspace $M_\pi$,  $\pi\in \widehat{G}$, into $M_{-\pi}$. Moreover, on the subspace $\aB \subseteq A$ the imap coincides with the usual adjoint, \ie, we have $i(b) = b^*$.
\end{defn}

The following lemma shows that the adjoints of left multiplication operators commute with right multiplication operators.

\begin{lemma}\label{inv VI}
For all $x\in M_\pi$, $y\in M_\rho$ and $z\in M_{\pi+\sigma}$ with $\pi,\rho,\sigma\in\widehat{G}$ we have
\begin{align*}
m(\lambda^+_x(z),y)=\lambda^+_x(m(z,y)).
\end{align*}
\end{lemma}
\begin{proof}
	It suffices to note that equation \eqref{eq:assoc} implies that 
\begin{align*}
\lambda_x\circ\Psi_{\sigma,\rho}=\Psi_{\pi+\sigma,\rho}\circ(\lambda_x\otimes_{\aB} \id_\rho).
\end{align*}
Indeed, taking adjoints then leads to 
\begin{align*}
\Psi^+_{\sigma,\rho}\circ \lambda_x^+=(\lambda^+_x\otimes_{\aB}\id_{\rho})\circ\Psi^+_{\pi+\sigma,\rho}
\end{align*}
which verifies the asserted formula since the maps $\Psi_{\sigma,\rho}$ and $\Psi_{\pi+\sigma,\rho}$ are unitary.
\end{proof}

\begin{thm}\label{inv VII}
	For all $x\in A$ we have $\lambda^+_x=\lambda_{i(x)}$.
\end{thm}
\begin{proof}
	It suffices to show the assertion for elements in individual direct summands. For this, let $x \in M_\pi$, $y \in M_\rho$, and $z \in M_\sigma$ with $\pi, \rho, \sigma \in \widehat{G}$. Then using Lemma \ref{inv VI} gives
	\begin{align*}	
		\scal{\lambda_{i(x)}(y), z}_{\aB}&=\scal{m(i(x),y), z}_{\aB}=\scal{m(\lambda^+_x(\one_{\aB}),y), z}_{\aB}	=\scal{\lambda^+_x(m(\one_{\aB},y), z}_{\aB}
		\\
		&= \scal{m(\one_{\aB},y), m(x,z)}_{\aB}=\scal{y,m(x,z)}_{\aB}=\scal{\lambda^+_x(y), z}_{\aB}.
		\qedhere
	\end{align*}
\end{proof}

We conclude this section with two useful corollaries, e.\,g., we finally verify that the map $i:A\rightarrow A$ from Definition \ref{inv IV} actually defines an involution.

\begin{cor}\label{inv VIII}
Let $P_0:A\rightarrow A$ be the canonical projection onto the subalgebra $\aB$. Then for all $x,y\in A$ we have 
\begin{align*}
\langle x,y\rangle_{\aB}=P_0(m(i(y),x)).
\end{align*}
\end{cor}
\begin{proof}
	Since the element $\one_{\aB}$ is fixed by $P_0$ we conclude from Theorem \ref{inv VII} that
	\begin{equation*}
		\langle x,y\rangle_{\aB}
		= \langle m(i(y),x),\one_{\aB}\rangle_{\aB}
		= \langle P_0( m(i(y),x)),\one_{\aB}\rangle_{\aB}
		= P_0(m(i(y),x)).
		\qedhere
	\end{equation*}
\end{proof}

\begin{cor}\label{inv IX}
The algebra $A$ is involutive, \ie, for all $x,y\in A$ we have
\begin{align*}
i(i(x))=x \quad \text{and} \quad i(m(x,y))=m(i(y),i(x)).
\end{align*}
\end{cor}
\begin{proof}
	Applying Theorem \ref{inv VII} twice gives
	\begin{gather*}
		\scal{i \bigl( i(x) \bigr),z}_{\aB}
		= \scal{\one_\aB, m(i(x),z)}_{\aB}
		= \scal{x,z}_{\aB},
		\\
		\scal{i(m(x,y)), z}_{\aB}
		= \scal{\one_\aB, m(m(x,y),z)}_{\aB}
		= \scal{i(x), m(y,z)}_{\aB}
		= \scal{m(i(y),i(x)),z}_{\aB}
	\end{gather*}
	for all $z \in A$ which in turn implies that $i(i(x)) = x$ and $i(m(x,y))=m(i(y),i(x))$. 
\end{proof}

\subsection{Construction of a Free Action}
\label{constr free action cmpt ab}

In the last subsection we turned $A = \bigoplus_{\pi \in \widehat{G}} M_\pi$ into a \Star algebra and right pre-Hilbert $\aB-\aB$-bimodule. We denote by $\bar A$ the corresponding completion of $A$ with respect to the norm
\begin{equation*}
	\norm{x}_2 := \norm{\scal{x,x}_{\aB}}^{1/2} = \norm{P_0(m(i(x),x))}^{1/2}
\end{equation*}
(\cf Corollary \ref{inv VIII}). By Corollary \ref{inv III}, left multiplication with an element $a \in A$ extends to an adjointable linear map on $\bar A$. Therefore, the map 
\begin{equation*}
	\lambda:A \to \mathcal L(\bar A), \quad a\mapsto\lambda_a
\end{equation*}
is well-defined. Moreover, the characterization of the norm implies that the vector \mbox{$\one_{\aB}\in\bar A$} is separating the operators $\lambda(A)\subseteq \mathcal L(\bar A)$, \ie, if $\lambda_a(\one_{\aB})=0$ for some $a\in A$ then $a=0$. The intention of this section is to finally construct a free C\Star dynamical system $(\aA,G,\alpha)$ with fixed point algebra~$\aB$. 

\begin{prop}\label{constr I}
The map $\lambda:A \to \mathcal L(\bar A)$, $a\mapsto \lambda_a$ is a faithful representation of the \Star algebra $A$ by adjointable operators on the right Hilbert $\aB-\aB$-bimodule $\bar A$. Moreover, its restriction to each $M_\pi$, $\pi\in\widehat{G}$, is isometric.
\end{prop}
\begin{proof}
The necessary algebraic conditions are easily checked using Corollary \ref{inv IX}. Moreover, the injectivity of the map $\lambda$ is a consequence of the previous discussion about the separating vector $\one_{\aB}\in\bar A$. To verify that the restriction of $\lambda$ to each $M_\pi$, $\pi\in\widehat{G}$, is isometric, we fix $\pi\in\widehat{G}$ and use inequality (\ref{cool inequality C}) of Proposition \ref{inv XI} which implies that $\norm{\lambda_y}^2_{\text{op}}\leq\norm{y}^2_2$ holds for all $y\in M_\pi$. On the other hand, the inequality 
\begin{align*}
\norm{y}^2_2=\norm{\lambda_y(\one_{\aB})}^2_2\leq\norm{\lambda_y}^2_{\text{op}}
\end{align*}
follows from the observation that $\one_{\aB}\in\bar A$ satisfies $\norm{\one_{\aB}}_2=1$. We conclude that $\norm{\lambda_y}_{\text{op}}=\norm{y}_2$ holds for each $y\in M_\pi$, which finally shows that the restriction of $\lambda$ to $M_\pi$ is isometric and thus completes the proof. 
\end{proof}

\begin{defn}\label{constr II}
	We denote by $\alg A$ the C\Star algebra which is generated by the image of $\lambda$, \ie, the closure of $\lambda(A)$ with respect to the operator norm on $\mathcal L(\bar A)$. In particular, we point out that $\alg A$ contains $A$ as a dense \Star subalgebra.
\end{defn}

To proceed we need to endow the C\Star algebra $\aA$ with a continuous action of the compact Abelian group $G$ by \Star automorphisms. For this purpose we first construct a strongly continuous unitary representation of $G$ on the right Hilbert $\aB-\aB$-bimodule $\bar A$.

\begin{lemma}\label{constr III}
	For each $\pi\in\widehat{G}$ the map $U_\pi:G\rightarrow U(M_\pi)$, $g\mapsto (U_{\pi})_g$ given by
	\begin{equation*}
	(U_{\pi})_g(x) := \pi_g\cdot x
	\end{equation*}
	is a strongly continuous unitary representation of $G$ on the right Hilbert $\aB-\aB$-bi\-mo\-dule $M_\pi$. Moreover, taking direct sums and continuous extensions then gives rise to a strongly continuous unitary representation $U:G\rightarrow U(\bar A)$, $g\mapsto U_g$ of $G$ on the right Hilbert $\aB-\aB$-bimodule $\bar A$.
\end{lemma}
\begin{proof}
The necessary computations are straightforward using the right $\aB$-valued inner products $\langle\cdot,\cdot\rangle_\pi$ and $\langle\cdot,\cdot\rangle$ of the spaces $M_\pi$ and $\bar A$, respectively.
\end{proof}

\begin{lemma}\label{constr IV}
	The map $\alpha:G\rightarrow\Aut(\aA)$, $g\mapsto\alpha_g$ given by 
	\begin{equation*}
		\alpha_g(\lambda_a) := U_g \, \lambda_a \, U_g^+
	\end{equation*}
	is a continuous action of $G$ on $\aA$ by \Star automorphisms.
\end{lemma}
\begin{proof}
	The action property is obviously satisfied due to the fact that $g \mapsto U_g$ is a representation (Lemma~\ref{constr III}). Continuity follows from the strong continuity of the map~$U$ from Lemma \ref{constr III}.
\end{proof}


We are finally ready to present the main result of this section.

\begin{thm}\label{fac sys into free}
	The C\Star dynamical system $(\aA,G,\alpha)$ associated to the factor system $(M_\pi, \Psi_{\pi, \rho})_{\pi, \rho\in\widehat{G}}$ is free and satisfies $A(\pi)=M_\pi$ for all $\pi\in\widehat{G}$. In particular, its fixed point algebra is given by $\aB$.
\end{thm}
\begin{proof}
	(i) Let $\pi\in\widehat{G}$. We first check that the corresponding isotypic component $A(\pi)$ is equal to $M_\pi$. Indeed, using the separating vector shows that $\alpha_g(a) = U_g(a)$ holds for elements $a \in A$. In particular, the elements of $M_\pi\subseteq \aA$ are contained in $A(\pi)$. Moreover, the continuity of the projection $P_\pi:\aA\rightarrow\aA$ onto $A(\pi)$ implies that $M_\pi=P_\pi(A)\subseteq \aA$ is dense in $A(\pi)$. Since the restriction of $\lambda$ to $M_\pi$ is isometric (\cf Lemma \ref{constr I}), we conclude that $M_\pi$ is closed in $A(\pi)$ and hence that $A(\pi)=M_\pi$ as claimed. In particular, the fixed point algebra of $(\aA,G,\alpha)$ is given by $\aB$.

	(ii) Next we show that the C\Star dynamical system $(\aA,G,\alpha)$ is free. For this purpose, we again fix $\pi\in\widehat{G}$. Since $A(\pi)=M_\pi$ holds by part (i) and 
	\begin{align*}
	M_{-\pi}\cdot M_\pi:=\Span\{m(x,y) \;|\; x\in M_{-\pi}, y\in M_\pi\}
	\end{align*}
	is dense in $\aB$ by construction, it follows that the multiplication map on $\aA$ induces an isomorphism of $\aB-\aB$-Morita equivalence bimodules between $A(-\pi)\widehat{\otimes}_{\aB}A(\pi)$ and $\aB$. We therefore conclude from Corollary \ref{equcondsatactgrpcomm} that the C\Star dynamical system $(\aA,G,\alpha)$ is free.
\end{proof}

\begin{rmk}
	\sloppy
	To emphasize the dependence on a given factor system $(M,\Psi):=(M_\pi, \Psi_{\pi, \rho})_{\pi, \rho\in\widehat{G}}$, we occasionally write $\bigr(\alg A_{(M,\Psi)}, G, \alpha_{(M,\Psi)}\bigl)$ for the associated free C\Star dynamical system from Theorem \ref{fac sys into free}.
\end{rmk}



\section{Classification of Free Actions of Compact Abelian Groups}
\label{sec:classification}

In the previous section we have seen how a factor system gives rise to a free C\Star dynamical system and vice versa. In this section we finally establish a classification theory for free actions of compact Abelian groups. If not mentioned otherwise, $\aB$ denotes a fixed unital C\Star algebra and $G$ a fixed compact Abelian group.

\begin{defn}~\label{equclasssatact} 
	\begin{enumerate}
	\item
	Let $(\aA,G,\alpha)$ and $(\aA',G,\alpha')$ be two free C\Star dynamical systems with unital C\Star algebras $\aA$ and $\aA'$ such that $\aA^G=(\aA')^G=\aB$. We call $(\aA,G,\alpha)$ and $(\aA',G,\alpha')$ \emph{equivalent} if there is a $G$-equivariant \Star isomorphism $T:\aA\rightarrow\aA'$ satisfying $T_{\mid_{\aB}}=\id_{\aB}$. 
	\item
	 We call two factor systems $(M_\pi, \Psi_{\pi, \rho})_{\pi, \rho\in\widehat{G}}$ and $(M'_\pi, \Psi'_{\pi, \rho})_{\pi, \rho\in\widehat{G}}$ \emph{equivalent} if there is a family $(T_\pi:M_\pi\rightarrow M'_\pi)_{\pi \in \widehat{G}}$ of Morita equivalence $\aB-\aB$-bimodule isomorphisms satisfying for all $\pi,\rho\in\widehat{G}$
	\begin{align}
		\Psi'_{\pi, \rho}\circ(T_\pi\otimes_{\aB} T_\rho)=T_{\pi+\rho}\circ\Psi_{\pi, \rho}.\label{fac sys eq}
	\end{align}
	\end{enumerate}
	\end{defn}

We are now in the position to state and prove on of the main classification theorems of this article.

\begin{thm}\label{fac sys equ}
Let $G$ be a compact Abelian group and $\aB$ a unital C\Star algebra. Furthermore, let $\varphi:\widehat{G}\rightarrow\Pic(\aB)$ be a group homomorphism and $(M_\pi, \Psi_{\pi, \rho})_{\pi, \rho\in\widehat{G}}$ and $(M'_\pi, \Psi'_{\pi, \rho})_{\pi, \rho\in\widehat{G}}$ factor systems for the map $\varphi$. Then the following statements are equivalent:
	\begin{equivalence}
	\item[\emph{(a)}]	\label{en:factor_sys_equiv cmpt ab}
		The factor systems are equivalent.
	\item[\emph{(b)}]	\label{en:dyn_sys_equiv cmpt ab}
		The associated free C\Star dynamical systems are equivalent.
	\end{equivalence}
\end{thm}
\begin{proof}
Suppose first that the factor systems $(M_\pi, \Psi_{\pi, \rho})_{\pi, \rho\in\widehat{G}}$ and $(M'_\pi, \Psi'_{\pi, \rho})_{\pi, \rho\in\widehat{G}}$ are equivalent and let $(T_\pi:M_\pi\rightarrow M'_\pi)_{\pi \in \widehat{G}}$ be a family of $\aB-\aB$-Morita equivalence bimodule isomorphisms such that equation (\ref{fac sys eq}) holds for all $\pi,\rho\in\widehat{G}$. Furthermore, let 
\begin{align*}
	A &:= \bigoplus_{\pi\in\widehat{G}}M_{\pi} 
	&
	&\text{and} 
	&
	A' &:= \bigoplus_{\pi\in\widehat{G}}M'_{\pi}
\end{align*}
be the corresponding $^*$-algebras with involutions given by $i$ and $i'$, respectively. Then a few moments thought shows that the direct sum of the maps $T_\pi:M_\pi\rightarrow M'_\pi$, $\pi\in\widehat{G}$, provides a $G$-equivariant $^*$-isomorphism $T:A\rightarrow A'$ of algebras. In fact, the map $T$ is clearly a $G$\ndash equivariant isomorphism of right pre-Hilbert $\aB-\aB$-bimodules by construction. Moreover, the assumption that equation (\ref{fac sys eq}) holds for all $\pi,\rho\in\widehat{G}$ implies that $T$ is multiplicative. By Theorem \ref{inv VII}, it is also $^*$-preserving, that is, $T(i(x))=i'(T(x))$ holds for all $x\in A$. Passing over to the continuous extension of $T$ provides a $G$-equivariant isomorphism $\bar T:\bar A\rightarrow \bar A'$ of right Hilbert $\aB-\aB$-bimodules and it is easily checked with the help of the previous discussion that the relation
\begin{align*}
\Ad[\bar T]\circ\lambda=\lambda'\circ T
\end{align*}
holds, where $\lambda:A \to \mathcal L(\bar A)$ and $\lambda':A' \to \mathcal L(\bar A')$ denote the faithful $^*$-representations from Proposition \ref{constr I}. In particular, we conclude that the map $\Ad[\bar T]:\mathcal L(\bar A)\rightarrow \mathcal L(\bar A')$ restricts to a $G$-equivariant $^*$-isomorphism between the associated free C\Star dynamical systems $(\aA,G,\alpha)$ and $(\aA',G,\alpha')$ which completes the first part of the proof.

Suppose, conversely, that the associated free C\Star dynamical systems $((\aA,m),G,\alpha)$ and $((\aA',m'),G,\alpha')$ are equivalent and let $T:\aA\rightarrow\aA'$ be a $G$-equivariant $^*$-isomorphism. Then it is a consequence of the $G$-equivariance of the map $T$ that the corresponding restriction maps $T_\pi:=T_{\mid M_\pi}:M_\pi\rightarrow M'_\pi$, $\pi\in\widehat{G}$, are well-defined and $\aB-\aB$ bimodule isomorphisms. Moreover, the multiplicativity of $T$ implies that equation (\ref{fac sys eq}) holds for all $\pi,\rho\in\widehat{G}$. Hence it remains to show that the family $(T_\pi:M_\pi\rightarrow M'_\pi)_{\pi\in\widehat{G}}$ preserves the $\aB$-valued inner products. To see that this is true, we first conclude from the $^*$-invariance of $T$ that $T_\pi(x)^*=T_{-\pi}(x^*)$ holds for all $\pi\in\widehat{G}$ and all $x\in M_\pi$. It follows from a short computation involving equation (\ref{fac sys eq}) that
\begin{align*}
\langle T_\pi(x),T_\pi(y)\rangle_\aB=m'(T_\pi(x),T_\pi(y)^*)=m'(T_\pi(x),T_{-\pi}(y^*))=m(x,y^*)=\langle x,y\rangle_\aB
\end{align*}
holds for all $\pi\in\widehat{G}$ and all $x,y\in M_\pi$. The corresponding computation for the left $\aB$-valued inner products can be verified in a similar way and completes the proof.
\end{proof}

\begin{defn}
We write $\Ext(\aB,G)$ for the set of equivalence classes of free actions of $G$ with fixed point algebra $\aB$. The equivalence class of a free C\Star dynamical system $(\aA, G, \alpha)$ with fixed point algebra $\aB$ is denoted by
$[(\aA, G, \alpha)]$.
\end{defn}

Recall that for a free C\Star dynamical system $(\aA, G, \alpha)$ with fixed point algebra $\aB$ we have a group homomorphism $\varphi_{\aA}:\widehat{G}\rightarrow\Pic(\aB)$ given by $\varphi_\aA(\pi):=[A(\pi)]$ (\cf Proposition \ref{prop:sat_vs_pic}). By Theorem \ref{fac sys equ}, the map $\varphi_{\aA}$ only depends on the equivalence class of $(\aA, G, \alpha)$ and hence we have an invariant
\begin{align*}
	I:\Ext(\aB,G)\rightarrow\Hom_{\text{gr}}\bigr(\widehat{G},\Pic(\aB)\bigl), \quad I\bigr([(\aA,G,\alpha)]\bigl):=\varphi_\aA.
\end{align*}
In particular, we may partition $\Ext(\aB, G)$ into the subsets
\begin{equation*}
	\Ext(\aB,G,\varphi) := I^{-1}(\varphi) = \{ [(\aA, G, \alpha)] \in \Ext(\aB, G) \;|\; \varphi_\aA = \varphi \}.
\end{equation*}
For a fixed group homomorphism $\varphi:\widehat{G}\rightarrow\Pic(\aB)$, set $\Ext(\aB, G, \varphi)$ may be empty. We postpone this problem until the end of the section and concentrate first on characterizing the set $\Ext(\aB, G, \varphi)$ and its C\Star dynamical systems.
We start with a useful statement about automorphisms of Morita equivalence bimodules. Although it might be well-known to experts, we have not found such a statement explicitly discussed in the literature.

\begin{prop}	\label{autME}
	Let $T$ be an automorphism of the Morita equivalence $\aB-\aB$-bimodule~$M$. Then there exists a unique unitary element $u$ of the center of $\aB$, \ie, an element $u\in UZ(\aB)$, such that $T(m)=u\cdot m$ for all $m\in M$. Furthermore, the map
	\begin{align*}
	\psi: UZ(\aB)\rightarrow\Aut_{\emph{ME}}(M), \quad \psi(u)(m):=u\cdot m,
	\end{align*}
	is an isomorphism of groups, where $\Aut_{\emph{ME}}(M)$ denotes the group of automorphisms of the Morita equivalence $\aB-\aB$-bimodule $M$.
\end{prop}
\begin{proof}
We divide the proof of this statement into two steps:

(i) In the first step we show that the assertion holds for the canonical Morita equivalence $\aB-\aB$-bimodule $\aB$. To see that this is true, we choose $u\in UZ(\aB)$ and note that the map $T_u:\aB\rightarrow\aB$, $b\mapsto u\cdot b$ defines an automorphism of the Morita equivalence $\aB-\aB$-bimodule $\aB$. In particular, the assignment
\[\psi_1:UZ(\aB)\rightarrow\Aut_{\text{ME}}(\aB), \quad u\mapsto T_u
\]is an isomorphism of groups. In fact, given $T\in \Aut_{\text{ME}}(\aB)$, a short calculation shows that $T$ is uniquely determined by $T(\one_{\aB})$ which is an element in $UZ(\aB)$. 

(ii) In the second step we show that Morita equivalence automorphisms of $\aB$ are in one-to-one correspondence with automorphisms of $M$. To begin with, we denote by $\overline{M}$ the conjugate module and recall that the map
\begin{align*}
\Psi:M\widehat{\otimes}_{\aB}\overline{M}\rightarrow\aB, \quad \Psi(m\otimes_{\aB}\overline{m'}):={}_{\aB}\langle m,m'\rangle
\end{align*}
for $m,m'\in M$ defines an isomorphism of Morita equivalence $\aB-\aB$-bimodules (\cf \cite[Proposition 3.28]{ReWi98}). Therefore, given an element $T\in \Aut_{\text{ME}}(M)$, it is not hard to check that the composition map $T_{\Psi}:=\Psi\circ(T\otimes_{\aB}\id_{\overline{M}})\circ\Psi^{-1}$ defines an automorphism of the Morita equivalence $\aB-\aB$-bimodule $\aB$. Next, we show that the map 
\begin{align*}
\psi_2:\Aut_{\text{ME}}(\aB)\rightarrow\Aut_{\text{ME}}(M), \quad \psi_2(T)(m):=T(\one_{\aB})\cdot m
\end{align*}
is an isomorphism of groups. In fact, we first note that $\psi_2$ is a well-defined and injective group homomorphism. Since $M$ is a full right Hilbert $\aB$-module, there is a finite set of elements $m_i, m_i'\in M$ $(1\leq i\leq n)$ such that $\sum^n_{i=1} {}_{\aB}\langle m_i,m_i'\rangle=\one_{\aB}$. The surjectivity of $\psi_2$ is then a consequence of the equation
\begin{align*}
\psi_2(T_{\Psi})(m)=T_{\Psi}(\one_{\aB})\cdot m=T(m)
\end{align*}
which holds for all $m\in M$. The assertion therefore follows from $\psi=\psi_2\circ\psi_1$.
\end{proof}

\begin{cor}\label{autMEcor}
Let $M$ be a Morita equivalence $\aB-\aB$-bimodule and $u\in UZ(\aB)$. Then there exists a unique element $\Phi_M(u)\in UZ(\aB)$ such that $\Phi_M(u) \cdot m=m \cdot u$ holds for all $m\in M$. Furthermore, the map
\begin{align*}
\Phi_M:UZ(\aB)\rightarrow UZ(\aB), \quad u\mapsto \Phi_M(u)
\end{align*}
is an automorphism of groups.
\end{cor}
\begin{proof}
The first assertion is an immediate consequence of Proposition \ref{autME} applied to the automorphism of $M$ defined by $m\mapsto m \cdot u$. That the map $\Phi_M$ is an automorphism of groups follows from a short calculation.
\end{proof}

\begin{prop}\label{module structure pic}
The map 
\begin{align*}
\Phi:\Pic(\aB)\rightarrow\Aut(UZ(\aB)), \quad [M]\mapsto\Phi_M
\end{align*}
is a group homomorphism.
\end{prop}
\begin{proof}
(i) We first show that $\Phi$ is well-defined. Therefore let $\Psi:M\rightarrow N$ be an isomorphism of Morita equivalence $\aB-\aB$-bimodules and $u\in UZ(\aB)$. Then 
\begin{align*}
\Phi_M(u) \cdot \Psi(m)=\Psi(\Phi_M(u) \cdot m)=\Psi(m \cdot u)=\Psi(m) \cdot u=\Phi_N(u) \cdot \Psi(m)
\end{align*}
holds for all $m\in M$ which implies that $\Phi_M=\Phi_N$.

(ii) To see that $\Phi$ is a group homomorphism, let $M$ and $N$ be Morita equivalence $\aB-\aB$-bimodules and $u\in UZ(\aB)$. Then 
\begin{align*}
&\Phi_{M\widehat{\otimes}_{\aB}N}(u) \cdot (m\otimes_{\aB} n)=(m\otimes_{\aB} n) \cdot u=m\otimes_{\aB}(n \cdot u)
=m\otimes_{\aB}(\Phi_N(u) \cdot n)\\
&=(m \cdot \Phi_N(u))\otimes_{\aB} n=(\Phi_M(\Phi_N(u)) \cdot m)\otimes_{\aB} n=(\Phi_M\circ\Phi_N)(u) \cdot (m\otimes_{\aB} n).
\end{align*}
holds for all $m\in M$ and $n\in N$ which shows that $\Phi_{M\widehat{\otimes}_{\aB}N}=\Phi_M\circ\Phi_N$.
\end{proof}

\begin{rmk}\label{G module structure}
We point out that the map $\Phi$ from Proposition \ref{module structure pic} induces a map
\begin{align*}
\Phi_*:\Hom_{\text{gr}}(\widehat{G},\Pic(\aB))\rightarrow\Hom_{\text{gr}}\bigr(\widehat{G},\Aut(UZ(\aB)\bigl), \quad \Phi_*(\varphi):=\Phi\circ\varphi.
\end{align*}
In particular, each $\varphi\in\Hom_{\text{gr}}(\widehat{G},\Pic(\aB))$ determines a $\widehat{G}$-module structure on $UZ(\aB)$ which enables us to make use of classical group cohomology. In fact, given an element $\varphi\in\Hom_{\text{gr}}(\widehat{G},\Pic(\aB))$, the cohomology groups 
\begin{align*}
H^n_{\varphi}\bigr(\widehat{G},UZ(\aB)\bigl):=H^n_{\Phi\circ\varphi}\bigr(\widehat{G},UZ(\aB)\bigl)
\end{align*}
are at our disposal (\cf \cite[Chapter IV]{Ma95}).
\end{rmk}

\begin{thm}\label{sat and cocycles}
Let $G$ be a compact Abelian group and $\aB$ a unital C\Star algebra. Furthermore, let $\varphi:\widehat{G}\rightarrow\Pic(\aB)$ be a group homomorphism with $\Ext(\aB,G,\varphi)\neq\emptyset$ and choose for all $\pi\in\widehat{G}$ a Morita equivalence $\aB-\aB$-bimodule $M_{\pi}\in\varphi(\pi)$ such that $M_0=\aB$. Then the following assertions hold:
	\begin{itemize}
	\item[\emph{(a)}]
		Each class in $\Ext(\aB,G,\varphi)$ can be represented by a free C\Star dynamical system of the form $\bigr(\alg A_{(M,\Psi)}, G, \alpha_{(M,\Psi)}\bigl)$.

	\item[\emph{(b)}]
		Any other free C\Star dynamical system $(\alg A_{(M,\Psi')}, G, \alpha_{(M,\Psi')})$ representing an element of $\Ext(\aB,G,\varphi)$ satisfies $\Psi'=\omega\Psi$ with $(\omega\Psi)_{\pi, \rho}:=\omega(\pi,\rho)\Psi_{\pi, \rho}$ for all $\pi,\rho\in\widehat{G}$ for some 2-cocycle
		\begin{align*}
		\omega\in Z^2_{\varphi}(\widehat{G},UZ(\aB)).
		\end{align*}

\item[\emph{(c)}]
The free C\Star dynamical systems $\bigr(\alg A_{(M,\Psi)}, G, \alpha_{(M,\Psi)}\bigl)$ and $\bigr(\alg A_{(M,\omega\Psi)}, G, \alpha_{(M,\omega\Psi)}\bigl)$ are equivalent if and only if
\begin{align*}
\omega\in B^2_{\varphi}(\widehat{G},UZ(\aB)).
\end{align*}
\end{itemize}
\end{thm}
\begin{proof}
(a) Let $(\aA,G,\alpha)$ be a free C\Star dynamical system representing an element in $\Ext(\aB,G,\varphi)$ and recall that $(\aA,G,\alpha)$ gives rise to a factor system for the map $\varphi$ of the form $(A(\pi),m_{\pi, \rho})_{\pi, \rho\in\widehat{G}}$ (\cf Remark \ref{rmk nice eq}). Then the assumption implies that there is a family $(T_\pi:M_\pi\rightarrow A(\pi))_{\pi \in \widehat{G}}$ of Morita equivalence $\aB-\aB$-bimodule isomorphisms which can be used to define another family $(\Psi''_{\pi, \rho}:M_\pi\widehat{\otimes}_{\aB} M_\rho\rightarrow M_{\pi+\rho})_{\pi, \rho\in\widehat{G}}$ of Morita equivalence $\aB-\aB$-bimodule isomorphisms by
\begin{align*}
\Psi''_{\pi, \rho}:=T^+_{\pi+\rho}\circ m_{\pi, \rho}\circ(T_\pi\otimes_{\aB}T_\rho).
\end{align*}
In particular, it is not hard to see that the later family gives rise to a factor system $(M_\pi, \Psi''_{\pi, \rho})_{\pi, \rho\in\widehat{G}}$ for the map $\varphi$ which is equivalent to $(A(\pi),m_{\pi, \rho})_{\pi, \rho\in\widehat{G}}$. Therefore, the assertion is finally a consequence of Theorem \ref{fac sys equ}.

(b) Let $\bigr(\alg A_{(M,\Psi')}, G, \alpha_{(M,\Psi')}\bigl)$ be any other free C\Star dynamical system representing an element of $\Ext(\aB,G,\varphi)$ and choose $\pi,\rho\in\widehat{G}$. Then Proposition \ref{autME} implies that the automorphism $\Psi'_{\pi, \rho}\circ\Psi^{-1}_{\pi, \rho}$ of the Morita equivalence $\aB-\aB$-bimodule $M_{\pi+\rho}$ provides a unique element $\omega(\pi,\rho)\in UZ(\aB)$ satisfying
\begin{align*}
\Psi'_{\pi, \rho}=\omega(\pi,\rho)\Psi_{\pi, \rho}.
\end{align*}
Moreover, it is easily seen that the corresponding map $\omega:\widehat{G}\times\widehat{G}\rightarrow UZ(\aB)$ is a normalized 2-cochain. To see that $\omega$ actually defines a 2-cocycle, \ie, an element in $Z^2_{\varphi}(\widehat{G},UZ(\aB))$, we repeatedly use the factor system condition equation (\ref{eq:assoc}) and Proposition \ref{module structure pic}. For example, we find that
\begin{align*}
\Psi'_{\pi, \rho+\sigma}\circ(\id_\pi\otimes_{\aB}\Psi'_{\rho,\sigma})&=\Psi'_{\pi, \rho+\sigma}\circ(\id_\pi\otimes_{\aB}\omega(\rho,\sigma)\Psi_{\rho,\sigma})
\\
&=\Psi'_{\pi, \rho+\sigma}\circ(\id_\pi\omega(\rho,\sigma)\otimes_{\aB}\Psi_{\rho,\sigma})
\\
&=\Psi'_{\pi, \rho+\sigma}\circ(\Phi_\pi(\omega(\rho,\sigma))\id_\pi\otimes_{\aB}\Psi_{\rho,\sigma})
\\
&=\Phi_\pi(\omega(\rho,\sigma))\Psi'_{\pi, \rho+\sigma}\circ(\id_\pi\otimes_{\aB}\Psi_{\rho,\sigma})
\\
&=\Phi_\pi(\omega(\rho,\sigma))\omega(\pi,\rho+\sigma)\Psi_{\pi, \rho+\sigma}\circ(\id_\pi\otimes_{\aB}\Psi_{\rho,\sigma})
\end{align*}
holds for all $\pi,\rho,\sigma\in\widehat{G}$, where $\id_\pi\omega(\rho,\sigma)=\Phi_\pi(\omega(\rho,\sigma))\id_\pi$ is understood in the sense of Corollary \ref{autMEcor}.

(c) If $\omega=d_{\varphi}h$ holds for some element $h\in C^1(\widehat{G},UZ(\aB))$, then the factor systems $(M_\pi, \Psi_{\pi, \rho})_{\pi, \rho\in\widehat{G}}$ and $(M_\pi,\omega(\pi,\rho)\Psi_{\pi, \rho})_{\pi, \rho\in\widehat{G}}$ are equivalent. Hence, the assertion follows from Theorem \ref{fac sys equ}. If, on the other hand, $(M_\pi, \Psi_{\pi, \rho})_{\pi, \rho\in\widehat{G}}$ and $(M_\pi,\omega(\pi,\rho)\Psi_{\pi, \rho})_{\pi, \rho\in\widehat{G}}$ are equivalent, then we conclude  from Proposition \ref{autME} that there exists an element $h\in C^1(\widehat{G},UZ(\aB))$ which implements the equivalence given by a family $(T_\pi)_{\pi \in \widehat{G}}$ of Morita equivalence $\aB-\aB$-bimodule isomorphisms $T_\pi:M_\pi\rightarrow M_\pi$, \ie, we have $T_\pi=T_{h(\pi)}$ for all $\pi\in\widehat{G}$. Moreover, a few moments thought shows that $\omega=d_{\varphi}h\in B^2_{\varphi}(\widehat{G},UZ(\aB))$.
\end{proof}

\begin{cor}\label{sat and cocycles cor}
    Let $G$ be a compact Abelian group and $\aB$ a unital C\Star algebra. Furthermore, let $\varphi:\widehat{G}\rightarrow\Pic(\aB)$ be a group homomorphism with $\Ext(\aB,G,\varphi)\neq\emptyset$. Then the map
	\begin{align*}
	H^2_{\varphi}(\widehat{G},UZ(\aB))\times\Ext(\aB,G,\varphi)&\rightarrow\Ext(\aB,G,\varphi)
	\\
	\Bigr(\bigr[\omega\bigl],\bigr[\bigr(\alg A_{(M,\Psi)}, G, \alpha_{(M,\Psi)}\bigl)\bigl]\Bigl)&\mapsto \bigr[\bigr(\alg A_{(M,\omega\Psi)}, G, \alpha_{(M,\omega\Psi)}\bigl)\bigl]
	\end{align*}
	is a well-defined simply transitive action.
\end{cor}

We conclude with a remark which shows that our constructions from Section \ref{prelud} are, up to isomorphisms, inverse to each other.

\begin{rmk}
	We first recall that each free C\Star dynamical system $(\aA,G,\alpha)$ with unital C\Star algebra $\aA$ and compact Abelian group $G$ gives rise to a factor system for the map $\varphi_{\aA}$ of the form $(A(\pi),m_{\pi, \rho})_{\pi, \rho\in\widehat{G}}$ (\cf Remark \ref{rmk nice eq}). Furthermore, it follows from Theorem \ref{fac sys equ} that the free C\Star dynamical system associated to this factor system is equivalent to $(\aA,G,\alpha)$. On the other hand, given a group homomorphism $\varphi:\widehat{G}\rightarrow\Pic(\aB)$ and a factor system $(M_\pi, \Psi_{\pi, \rho})_{\pi, \rho \in \widehat{G}}$ for the map $\varphi$, it is easily seen that the associated free C\Star dynamical system $\bigr(\alg A_{(M,\Psi)}, G, \alpha_{(M,\Psi)}\bigl)$ in Theorem~\ref{fac sys into free} recovers the original factor system $(M_\pi, \Psi_{\pi, \rho})_{\pi, \rho\in\widehat{G}}$. Indeed, Theorem \ref{fac sys into free} shows that $\alg A_{(M,\Psi)}(\pi)=M_\pi$ holds for all $\pi\in\widehat{G}$. Moreover, the multiplication map of $\alg A_{(M,\Psi)}$ is by construction uniquely determined by the factor system, \ie, we have $m_{\pi,\rho}=\Psi_{\pi, \rho}$ for all $\pi,\rho\in\widehat{G}$. We therefore conclude that our constructions, \ie, the procedure of associating a free C\Star dynamical system to a factor system and vice versa, are, up to isomorphisms, inverse to each other:
	\begin{align*}
		(\aA,G,\alpha) 
		&\quad \xmapsto{\quad \text{F.S.} \quad} \quad 
		(A(\pi),m_{\pi,\rho})_{\pi, \rho\in\widehat{G}}.
		\\
		(M_\pi, \Psi_{\pi, \rho})_{\pi, \rho\in\widehat{G}}
		&\quad \xmapsto{\quad \text{C$^*$} \quad} \quad
		\bigr(\alg A_{(M,\Psi)}, G, \alpha_{(M,\Psi)}\bigl)
	\end{align*}
\end{rmk}

\subsection*{Non-emptiness of $\Ext(\aB,G,\varphi)$}
\label{non-emptiness}

As we have already discussed before, each group homomorphism $\varphi:\widehat{G}\rightarrow\Pic(\aB)$ gives rise to both a family $(M_{\pi})_{\pi\in\widehat{G}}$ of Morita equivalence $\aB-\aB$-bimodules and a \mbox{family} 
\begin{align*}
(\Psi_{\pi,\rho}:M_{\pi}\widehat{\otimes}_{\aB}M_{\rho}\rightarrow M_{\pi+\rho})_{\pi,\rho\in\widehat{G}}
\end{align*}
of Morita equivalence $\aB-\aB$-bimodules isomorphisms. Given a group homomorphism $\varphi:\widehat{G}\rightarrow\Pic(\aB)$ and such a
family $(M_{\pi},\Psi_{\pi,\rho})_{\pi,\rho\in\widehat{G}}$ satisfying $M_0=\aB$, $\Psi_{0,0}=\id_{\aB}$ and $\Psi_{\pi,0}=\Psi_{0,\pi}=\id_\pi$ for all $\pi\in\widehat{G}$ (which need not be a factor system), we can examine for all $\pi,\rho,\sigma\in\widehat{G}$ the automorphism
\begin{align*}
d_M\Psi(\pi,\rho,\sigma):=\Psi_{\pi + \rho, \sigma}\circ(\Psi_{\pi, \rho} \tensor_{\aB} \id_\sigma)\circ(\id_\pi \tensor_{\aB} \Psi^+_{\rho, \sigma})\circ\Psi^+_{\pi, \rho +\sigma}
\end{align*}
of the Morita equivalence $\aB-\aB$-bimodule $M_{\pi+\rho+\sigma}$. The family of all such maps $(d_M\Psi(\pi,\rho,\sigma))_{\pi,\rho,\sigma\in\widehat{G}}$ can be interpreted as an obstruction to the associativity of the multiplication (\cf Proposition \ref{thm:factor_sys cptab}). On the other hand, it follows from the construction and from Proposition \ref{autME} that the map $d_M\Psi$ can also be considered as a normalized $UZ(\aB)$-valued 3-cochain on $\widehat{G}$, \ie, as an element in $C^3(\widehat{G},UZ(\aB))$. In fact, even more is true:

\begin{lemma}
The map $d_M\Psi$ defines an element in $Z^3_{\varphi}(\widehat{G},UZ(\aB))$. 
\end{lemma}
\begin{proof}
For the sake of brevity we omit the lengthy calculation at this point and refer instead to \cite[Lemma 1.10 (5)]{Ne06}.
\end{proof}

\begin{lemma}\label{independent of choices}
The class $[d_M\Psi]$ in $H^3_{\varphi}(\widehat{G},UZ(\aB))$ is independent of all choices made.
\end{lemma}
\begin{proof}
(i) We first show that the class $[d_M\Psi]$ is independent of the choice of the family $(\Psi_{\pi,\rho})_{\pi,\rho\in\widehat{G}}$. Therefore, let $(\Psi'_{\pi,\rho})_{\pi,\rho\in\widehat{G}}$ be another choice and note that Proposition \ref{autME} implies that there exists an element $h\in C^2(\widehat{G},UZ(\aB))$ satisfying $\Psi'_{\pi,\rho}=h(\pi,\rho)\Psi_{\pi,\rho}$ for all $\pi,\rho\in\widehat{G}$. A~short calculation then shows that
\begin{align*}
\Psi'_{\pi + \rho, \sigma}\circ(\Psi'_{\pi, \rho} \tensor_{\aB} \id_\sigma)=h(\pi+\rho,\sigma)h(\pi,\rho)d_M\Psi(\pi,\rho,\sigma)(\Psi_{\pi,\rho+\sigma}\circ(\id_\pi\tensor_{\aB}\Psi_{\rho,\sigma}))
\end{align*}
holds for all $\pi,\rho\sigma\in\widehat{G}$. On the other hand, it follows from Proposition \ref{module structure pic} that
\begin{align*}
&d_M\Psi'(\pi,\rho,\sigma)(\Psi'_{\pi,\rho+\sigma}\circ(\id_\pi\tensor_{\aB}\Psi'_{\rho,\sigma}))
\\
=&d_M\Psi'(\pi,\rho,\sigma)h(\pi,\rho+\sigma)\pi.h(\rho,\sigma)(\Psi_{\pi,\rho+\sigma}\circ(\id_\pi\tensor_{\aB}\Psi_{\rho,\sigma}))
\end{align*}
holds for all $\pi,\rho,\sigma\in\widehat{G}$. From these observations we can now easily conclude that the 3-cocycles $d_M\Psi'$ and $d_M\Psi$ are cohomologous.

(ii) As a second step, we show that the class $[d_M\Psi]$ does not dependent on the choice of the family $(M_\pi)_{\pi\in\widehat{G}}$. For this purpose, let $(M'_\pi)_{\pi\in\widehat{G}}$ be another choice and note that the construction implies that there is a family $(T_\pi:M_\pi\rightarrow M'_\pi)_{\pi \in \widehat{G}}$ of Morita equivalence $\aB-\aB$-bimodule isomorphisms. This family can now be used to define another family $(\Psi''_{\pi, \rho}:M'_\pi\widehat{\otimes}_{\aB} M'_\rho\rightarrow M'_{\pi+\rho})_{\pi, \rho\in\widehat{G}}$ of Morita equivalence $\aB-\aB$-bimodule isomorphisms~by
\begin{align*}
\Psi'_{\pi,\rho}:=T_{\pi+\rho}\circ\Psi_{\pi,\rho}\circ(T^+_\pi\otimes_{\aB}T^+_\rho).
\end{align*}

Then an explicit computation shows that
\begin{align*}
	d_M\Psi'(\pi,\rho,\sigma)
	&= \Psi'_{\pi + \rho, \sigma}\circ(\Psi'_{\pi, \rho} \tensor_{\aB} \id_\sigma)\circ(\id_\pi \tensor_{\aB} \Psi'^+_{\rho, \sigma})\circ\Psi'^+_{\pi, \rho +\sigma}
	\\
	&= T_{\pi+\sigma+\rho}\circ\Psi_{\pi+\rho,\sigma}\circ(T^+_{\pi+\rho}\otimes_{\aB}T^+_{\sigma})
	\\
	&\qquad \circ(T_{\pi+\rho}\otimes_{\aB}\id_{\sigma})\circ(\Psi_{\pi,\rho}\otimes_{\aB}\id_{\sigma})\circ(T^+_\pi\otimes_{\aB}T^+_\rho\otimes_{\aB}\id_{\sigma})
	\\
	&\qquad \circ(\id_\pi\otimes_{\aB}T_\rho\otimes_{\aB}T_\sigma)\circ(\id_\pi\otimes_{\aB}\Psi^+_{\rho,\sigma})\circ(\id_\pi\otimes_{\aB}T^+_{\rho+\sigma})
	\\
	&\qquad \circ(T_\pi\otimes_{\aB}T_{\rho+\sigma})\circ\Psi^+_{\pi,\rho+\sigma}\circ T^+_{\pi+\rho+\sigma}
	\\
	&= d_M\Psi(\pi,\rho,\sigma)
\end{align*}
holds for all $\pi,\rho,\sigma\in\widehat{G}$. We conclude that $d_M\Psi'=d_M\Psi$, \ie, that the 3-cocycle $d_M\Psi$ is unchanged.
\end{proof}

\begin{defn}
Let $\varphi:\widehat{G}\rightarrow\Pic(\aB)$ be a group homomorphism. We call 
\begin{align*}
\chi(\varphi):=[d_M\Psi]\in H^3_{\varphi}(\widehat{G},UZ(\aB))
\end{align*}
the \emph{characteristic class} of $\varphi$.
\end{defn}

The following result provides a group theoretic criterion for the non-emptiness of the set $\Ext(\aB,G,\varphi)$.

\begin{thm} \label{classsatactgrpcomm}
Let $G$ be a compact Abelian group and $\aB$ a unital C\Star algebra. Furthermore, let $\varphi:\widehat{G}\rightarrow\Pic(\aB)$ be a group homomorphism. Then $\Ext(\aB,G,\varphi)$ is non-empty if and only if the class $\chi(\varphi)\in H_{\varphi}^3(\widehat{G},UZ(\aB))$ vanishes. 
\end{thm} 
\begin{proof}
($\Rightarrow$) Suppose first that $\Ext(\aB,G,\varphi)$ is non-empty and let $(\aA,G,\alpha)$ be a free C\Star dynamical system representing an element in $\Ext(\aB,G,\varphi)$. Then $(\aA,G,\alpha)$ gives rise to a factor system for the map $\varphi$ of the form $(A(\pi),m_{\pi, \rho})_{\pi, \rho\in\widehat{G}}$ (\cf Remark \ref{rmk nice eq}) and the associativity of the multiplication implies that the corresponding characteristic class $\chi(\varphi)\in H_{\varphi}^3(\widehat{G},UZ(\aB))$ vanishes. 

($\Leftarrow$) Let $(M_{\pi})_{\pi\in\widehat{G}}$ be a family of Morita equivalence $\aB-\aB$-bimodules and 
\begin{align*}
(\Psi_{\pi,\rho}:M_{\pi}\widehat{\otimes}_{\aB}M_{\rho}\rightarrow M_{\pi+\rho})_{\pi,\rho\in\widehat{G}}
\end{align*}
a family of Morita equivalence $\aB-\aB$-bimodules isomorphisms as described in the introduction. Furthermore, suppose, conversely, that the class 
\begin{align*}
\chi(\varphi)=[d_M\Psi]\in H_{\varphi}^3(\widehat{G},UZ(\aB))
\end{align*}
vanishes. Then there exists an element $h\in C^2(\widehat{G},UZ(B))$ with $d_M\Psi=d_{\varphi}h^{-1}$ which can be use to define a family 
$(\Psi'_{\pi,\rho}:M_\pi\widehat{\otimes}_{\aB} M_\rho\rightarrow M_{\pi+\rho})_{\pi,\rho\in\widehat{G}}$ of Morita equivalence bimodule $\aB-\aB$-isomorphism by 
\begin{align*}
\Psi'_{\pi,\rho}:=h(\pi,\rho)\Psi_{\pi,\rho}.
\end{align*}
The construction implies that $d_M\Psi'=\one_{\aB}$. In particular, it follows that $(M_\pi, \Psi'_{\pi, \rho})_{\pi, \rho\in\widehat{G}}$ is a factor system for the map $\varphi$ and we can finally conclude from Theorem \ref{fac sys into free} that the set $\Ext(\aB,G,\varphi)$ is non-empty.
\end{proof}

\begin{rmk}\label{inv elts}
The group of outer automorphisms $\Out(\aB)$ is always a subgroup of the Picard group $\Pic(\aB)$ (\cf Remark \ref{out_in_pic}). The intention of this remark is to describe the elements of the set $\Ext(\aB,G,\varphi)$ for a given group homomorphism $\varphi:\widehat{G}\rightarrow\Out(\aB)$. For this purpose, let $(\aA,G,\alpha)$ be a free C\Star dynamical system representing an element of $\Ext(\aB,G,\varphi)$. Then it is not hard to see that each isotypic component contains an invertible element in~$\aA$. In fact, it follows from Corollary \ref{equcondsatactgrpcomm} that the map 
\begin{align}
A(-\pi)\widehat{\otimes}_{\aB}A(\pi)\rightarrow\aB, \quad x\otimes_{\aB} y\mapsto xy\label{can multip}
\end{align}
is an isomorphism of Morita equivalence $\aB-\aB$-bimodules for all $\pi\in\widehat{G}$. Moreover, the assumption on $\varphi$ implies that for each $\pi\in\widehat{G}$ we have $\varphi(\pi)=[A(\pi)]$; that is, there is an automorphism $S(\pi)\in\Aut(\aB)$ and an isomorphism $T_\pi:M_{S(\pi)}\rightarrow A(\pi)$ of Morita equivalence $\aB-\aB$-bimodules. If we now define $u_\pi:=T_\pi(\one_{\aB})$, then a few moments thought shows that 
\begin{align*}
A(\pi)=u_\pi\aB=\aB u_\pi,
\end{align*}
from which we conclude together with equation (\ref{can multip}) that 
\begin{align*}
u_\pi\aB u_{-\pi}=u_{-\pi}\aB u_\pi=\aB.
\end{align*}
Consequently, the element $u_\pi\in A(\pi)$ is invertible in $\aA$. Conversely, let $(\aA,G,\alpha)$ be a C\Star dynamical system such that each isotypic component contains an invertible element. Then it is easily verified that $(\aA,G,\alpha)$ is free and that the corresponding group homomorphism $\varphi_{\aA}:\widehat{G}\rightarrow\Pic(\aB)$ from Proposition \ref{prop:sat_vs_pic} takes values in $\Out(\aB)$. Indeed, if for every $\pi \in \widehat{G}$ we have an element $u_\pi \in A(\pi)$ that is invertible in $\aA$, then $A(\pi)=u_\pi\aB=\aB u_\pi$, and the map 
\ref{can multip} is an isomorphism of Morita equivalence $\aB-\aB$-bimodules. In particular, $(\aA,G,\alpha)$ represents an element in $\Ext(\aB,G,\varphi_{\aA})$.

C\Star dynamical systems with the property that each isotypic component contains invertible elements have been studied, for example, in \cite{Sch04,Wa11b,Wa11e,Wass89} and may be considered as a noncommutative version of trivial principal bundles (\cf Remark \ref{tncpb rmk}).
\end{rmk}

\begin{rmk}\label{factor systems TNCPB}
The aim of the following discussion is to explain how to classify the C\Star dynamical systems described in Remark \ref{inv elts}. Indeed, let $(\aA,G,\alpha)$ be a C\Star dynamical system such that each isotypic component contains an invertible element. Furthermore, let $(u_\pi)_{\pi\in\widehat{G}}$ be a family of unitaries with $u_\pi\in A(\pi)$ and $u_0=\one_{\aB}$. Then the maps $S:\widehat{G}\rightarrow\Aut(\aB)$ and $\omega:\widehat{G}\times\widehat{G}\rightarrow U(\aB)$ given by 
\begin{align*}
S(\pi)(b):=u_{\pi}bu_{\pi}^* \quad \text{and} \quad \omega(\pi,\sigma):=u_{\pi}u_{\rho}u_{\pi+\rho}^*,
\end{align*}
give rise to an element $(S,\omega)\in C^1(\widehat{G},\Aut(\aB))\times C^2(\widehat{G},U(\aB))$
satisfying for all $\pi,\rho,\sigma\in\widehat{G}$ and $b\in\aB$ the relations
\begin{align}
S(\pi)(\omega(\rho,\sigma))\omega(\pi,\rho+\sigma)=\omega(\pi+\rho,\sigma)\omega(\pi,\rho)\label{fac sys grp I}
\\
S(\pi)(S(\rho)(b))=\omega(\pi,\rho)S(\pi+\rho)(b)\omega(\pi,\rho)^*.\label{fac sys grp II}
\end{align}
Conversely, each pair $(S,\omega)\in C^1(\widehat{G},\Aut(\aB))\times C^2(\widehat{G},U(\aB))$ satisfying, for all $\pi,\rho,\sigma\in\widehat{G}$ and $b\in\aB$, the relations (\ref{fac sys grp I}) and (\ref{fac sys grp II}) defines a factor system $(M_\pi)_{\pi\in\widehat{G}}$ and $(\Psi_{\pi,\rho})_{\pi,\rho\in\widehat{G}}$ given by $M_\pi:=M_{S(\pi)}$ and 
\begin{align*}
\Psi_{\pi,\rho}(b\otimes_{\aB} b'):=bS(\pi)(b')\omega(\pi,\rho).
\end{align*}
The associated free C\Star dynamical system represents an element of $\Ext(\aB,G,\varphi)$ with 
$\varphi:=\pr_{\aB}\circ S:\widehat{G}\rightarrow\Out(\aB)$, where $\pr_{\aB}:\Aut(\aB)\rightarrow\Out(\aB)$ denotes the canonical quotient homomorphism. It is worth pointing out that in this situation, the involution can be expressed explicitly in terms of the pair $(S,\omega)$ (\cf \cite[Construction A24]{Wa11e}). 

\end{rmk}

\section{Principal Bundles and Group Cohomology}
\label{sec:principal_bdl}

Let $P$ and $X$ be compact spaces. Furthermore, let $G$ be a compact group. Each locally trivial principal bundle $(P,X,G,q,\sigma)$ can be considered as a geometric object that is glued together from local pieces which are trivial, \ie, which are of the form $U\times G$ for some small open subset $U$ of $X$. This approach immediately leads to the concept of $G$-valued cocycles and therefore to a cohomology theory, called the \v Cech cohomology for the pair $(X,G)$. This cohomology theory gives a complete classification of locally trivial principal bundles with structure group $G$ and base space $X$ (see \cite{To00}). On the other hand, Theorem \ref{sat=free/comm} implies that each locally trivial principal bundle $(P,X,G,q,\sigma)$ gives rise to a free C\Star dynamical system $(C(P),G,\alpha_{\sigma})$ and it is therefore natural to ask how the \v Cech cohomology for the pair $(X,G)$ is related to our previous classification theory. But since our construction in Section \ref{prelud} is global in nature, it is not obvious how to encode local triviality in our factor system approach (though we recall that in the smooth category there is a one-to-one correspondence between free actions and locally trivial principal bundles). For this reason we now focus our attention on topological principal bundles, a notion of principal bundles which need not be locally trivial (\cf \cite[Remark 4.68]{ReWi98}).

\begin{defn}~\label{def: top bundles}
	\begin{enumerate}
		\item
		Let $P$ be a compact space and $G$ a compact group. We call a continuous action $\sigma:P\times G\rightarrow P$ which is free a \emph{topological principal bundle} or, more precisely, a \emph{topological principal $G$-bundle over $X:=P/G$}.
		\item
		We call two topological principal bundles $\sigma:P\times G\rightarrow P$ and $\sigma':P'\times G\rightarrow P'$ over $X$ \emph{equivalent} if there is a $G$-equivariant homeomorphism $h:P\rightarrow P'$ such that the inducex map on $X$ is the identity. 
	\end{enumerate}
\end{defn}


We now come back to our C\Star algebraic setting. Let $G$ be a compact Abelian group and $\sigma:P\times G\rightarrow P$ a topological principal $G$-bundle over $X$. Then for $\pi \in \widehat G$ the isotypic component $C(P)(\pi)$ is a finitely generated and projective $C(X)$-module according to Theorem~\ref{isononcommvec}. Therefore, the Theorem of Serre and Swan (\cf \cite{Swa62}) gives rise to a locally trivial complex line bundle $\mathbb{V}_\pi$ over $X$ such that $C(P)(\pi)$ as a right $C(X)$-module is isomorphic to the corresponding space $\Gamma\mathbb{V}_\pi$ of continuous sections. We point out that this isomorphism can be extended to an isomorphism between Morita equivalence $C(X)-C(X)$-bimodules. In what follows we identify the class of $C(P)(\pi)$ in $\Pic(C(X))$ with the class of the corresponding complex line bundle $\mathbb{V}_\pi$ in $\Pic(X)$. Then the group homomorphism $\varphi: \widehat{G} \rightarrow \Pic(C(X))$ induced by the topological principal bundle $\sigma:P\times G\rightarrow P$ is given by
\begin{align*}
	\varphi(\pi) = [\mathbb{V}_\pi] \in \Pic(X) \subseteq \Pic(C(X))
\end{align*}
(\cf Example~\ref{expl:pic_grp} and Proposition \ref{prop:sat_vs_pic}). In particular, the $\widehat{G}$-module structure on the unitary group $U(C(X))=C(X, \mathbb{T})$ induced by this group homomorphism (\cf Remark~\ref{G module structure}) is trivial since the left and right action of $C(X)$ on $C(P)(\pi)$ commute. At this point it is useful to introduce the following definition. 


\begin{defn}
Let $G$ be a compact Abelian group and $X$ a compact space. Furthermore, let $\varphi:\widehat{G}\rightarrow\Pic(X)$ be a group homomorphism. We denote by $\Ext_{\text{top}}(X,G,\varphi)$ the subset of $\Ext(C(X),G,\varphi)$ describing the topological principal $G$-bundles over $X$ inducing the map $\varphi$.
\end{defn}



We continue with a compact space $X$ and a group homomorphism $\varphi:\widehat{G} \to \Pic(X)\subseteq\Pic(C(X))$. Given a factor system $(M_\pi, \Psi_{\pi, \rho})_{\pi, \rho\in\widehat{G}}$ for the map $\varphi$, the canonical flip
\begin{align*}
\text{fl}_{\pi, \rho}:M_\pi\widehat{\otimes}_{C(X)}M_\rho\rightarrow M_\rho\widehat{\otimes}_{C(X)}M_\pi, \quad x\otimes_{C(X)}y\mapsto y\otimes_{C(X)}x
\end{align*}
defines an isomorphism of Morita equivalence $C(X)-C(X)$-bimodules for all $\pi,\rho\in\widehat{G}$. In particular, we may examine, for all $\pi,\rho\in\widehat{G}$, the automorphism
\begin{equation*}
	\mathcal{C}_M\Psi(\pi,\rho):=\Psi_{\rho,\pi}\circ\text{fl}_{\pi,\rho}\circ\Psi^+_{\pi,\rho}
\end{equation*}
of the Morita equivalence $C(X)-C(X)$-bimodule $M_{\pi+\rho}$. According to Proposition \ref{autME}, the map $\mathcal{C}_M\Psi$ can be considered as a normalized 2-cochain on $\widehat{G}$ with values in $C(X,\mathbb{T})$, \ie, as an element in $C^2(\widehat{G},C(X,\mathbb{T}))$. In fact, even more is true:

\begin{lemma}\label{comm II}
	The map $\mathcal{C}_M\Psi$ defines an antisymmetric element in $Z^2(\widehat{G},C(X,\mathbb{T}))$. 
\end{lemma}
\begin{proof}
	It is obvious that the map $\mathcal{C}_M\Psi$ satisfies $\mathcal{C}_M\Psi(\rho,\pi)=\mathcal{C}_M\Psi(\pi,\rho)^*$ for all $\pi,\rho\in\widehat{G}$. In order to verify that $\mathcal{C}_M\Psi$ defines an element in $Z^2(\widehat{G},C(X,\mathbb{T}))$ we have to show that
	\begin{equation}	\label{eq:2-cocycle}
		\mathcal{C}_M\Psi(\rho,\sigma)\mathcal{C}_M\Psi(\pi,\rho+\sigma)=\mathcal{C}_M\Psi(\pi+\rho,\sigma)\mathcal{C}_M\Psi(\pi,\rho)
	\end{equation}
	holds for all $\pi,\rho,\sigma\in\widehat{G}$. Indeed, explicit computations using the factor system property, equation~\eqref{eq:assoc}, show that
	\begin{gather*}
		\Psi_{\sigma+\rho,\pi}\circ(\Psi_{\sigma,\rho}\otimes_{C(X)}\id_{\pi})=\mathcal{C}_M\Psi(\rho,\sigma)\mathcal{C}_M\Psi(\pi,\rho+\sigma)\Psi_{\pi,\sigma+\rho}\circ(\id_{\pi}\otimes_{C(X)}\Psi_{\rho,\sigma})
		\shortintertext{and}
		\Psi_{\sigma,\rho+\pi}\circ(\id_{\sigma}\otimes_{C(X)}\Psi_{\rho,\pi})=\mathcal{C}_M\Psi(\pi,\rho)\mathcal{C}_M\Psi(\pi+\rho,\sigma)\Psi_{\pi+\rho,\sigma}\circ(\Psi_{\pi,\rho}\otimes_{C(X)}\id_{\pi})
	\end{gather*}
	hold for all $\pi,\rho,\sigma\in\widehat{G}$. Again using equation \eqref{eq:assoc} to the right-hand side of the later expression finally leads to the desired 2-cocycle condition \eqref{eq:2-cocycle}.
\end{proof}

\begin{lemma}\label{comm IV}
	For equivalent factor systems $(M_\pi, \Psi_{\pi, \rho})_{\pi, \rho\in\widehat{G}}$ and $(M'_\pi, \Psi'_{\pi, \rho})_{\pi, \rho\in\widehat{G}}$ we have
	\begin{align*}
	\mathcal{C}_{M'}\Psi'=\mathcal{C}_{M}\Psi.
	\end{align*}
\end{lemma}
\begin{proof}
By assumption there is a family $(T_\pi:M_\pi\rightarrow M'_\pi)_{\pi \in\widehat{G}}$ of Morita equivalence $C(X)-C(X)$-bimodule isomorphisms satisfying 
\begin{align*}
\Psi'_{\pi,\rho}:=T_{\pi+\rho}\circ\Psi_{\pi,\rho}\circ(T^+_\pi\otimes_{C(X)}T^+_\rho).
\end{align*}
Therefore, an explicit computation shows that for all $\pi, \rho \in \widehat G$ we have
\begin{align*}
	\mathcal{C}_{M'}\Psi'(\pi,\rho)&=\Psi'_{\rho,\pi}\circ\text{fl}'_{\pi,\rho}\circ\Psi'^+_{\pi,\rho}
	\\
	&=T_{\rho+\pi}\circ\Psi_{\rho,\pi}\circ(T^+_\rho\otimes_{C(X)}T^+_\pi)\circ\text{fl}'_{\pi,\rho}\circ(T_\pi\otimes_{C(X)}T_\rho)\circ\Psi^+_{\pi,\rho}\circ T^+_{\pi+\rho}
	\\
	&=T_{\rho+\pi}\circ\Psi_{\rho,\pi}\circ\text{fl}_{\pi,\rho}\circ\Psi^+_{\pi,\rho}\circ T^+_{\pi+\rho}=\mathcal{C}_M\Psi(\pi,\rho)
	\qedhere
\end{align*}
\end{proof}

\begin{lemma}\label{comm III}
Let $(M_\pi, \Psi_{\pi, \rho})_{\pi, \rho\in\widehat{G}}$ and $(M'_\pi, \Psi'_{\pi, \rho})_{\pi, \rho\in\widehat{G}}$ be two factor systems for the map $\varphi$. Then there exists an element $\omega\in Z^2(\widehat{G},C(X,\mathbb{T}))$ satisfying for all $\pi,\rho\in\widehat{G}$
\begin{align*}
\mathcal{C}_{M'}\Psi'(\pi,\rho)=\omega(\pi,\rho)\omega(\rho,\pi)^*\mathcal{C}_{M}\Psi(\pi,\rho).
\end{align*}

\end{lemma}
\begin{proof}
To verify the assertion we use Theorem \ref{sat and cocycles}, which implies that 
the factor system $(M'_\pi, \Psi'_{\pi, \rho})_{\pi, \rho\in\widehat{G}}$ is equivalent to a factor system of the form $(M_\pi, \omega(\pi,\rho)\Psi_{\pi, \rho})_{\pi, \rho\in\widehat{G}}$ for some 2-cocycle $\omega\in Z^2(\widehat{G},C(X,\mathbb{T}))$. In particular, we conclude from Lemma~\ref{comm IV} and a short calculation that
\begin{align*}
\mathcal{C}_{M'}\Psi'(\pi,\rho)=\mathcal{C}_M(\omega\Psi)(\pi,\rho)&=(\omega\Psi)_{\rho,\pi}\circ\text{fl}_{\pi,\rho}\circ(\omega\Psi)^+_{\pi,\rho}=\omega(\rho,\pi)\omega(\pi,\rho)^*\mathcal{C}_M\Psi(\pi,\rho)
\end{align*}
holds for all $\pi,\rho\in\widehat{G}$.
\end{proof}

Let us denote by $\Alt^2(\widehat{G},C(X,\mathbb{T}))$ the group of biadditive maps $\widehat{G}\times\widehat{G}\rightarrow C(X,\mathbb{T})$ that vanish on the diagonal. Then a~few moments thought shows that each \mbox{element} $\omega\in Z^2(\widehat{G},C(X,\mathbb{T}))$ gives rise to an element $\lambda_{\omega}\in\Alt^2(\widehat{G},C(X,\mathbb{T}))$ defined~by
\begin{equation*}
	\lambda_{\omega}(\pi,\rho):=\omega(\pi,\rho) \, \omega(\rho,\pi)^*,
\end{equation*}
which only depends on the class $[\omega]\in H^2(\widehat{G},C(X,\mathbb{T}))$. In particular, we obtain a group homomorphism
\begin{align*}
	\lambda:H^2(\widehat{G},C(X,\mathbb{T}))\rightarrow\Alt^2(\widehat{G},C(X,\mathbb{T})), \quad [\omega]\mapsto\lambda_{\omega}
\end{align*}
whose kernel is given by the subgroup $H^2_{\text{ab}}(\widehat{G},C(X,\mathbb{T}))$ of $H^2(\widehat{G},C(X,\mathbb{T}))$ describing the Abelian extensions of $\widehat{G}$ by $C(X,\mathbb{T})$. We recall from \cite[Proposition II.3]{Ne07} that the corresponding short exact sequence
\[0\longrightarrow H^2_{\text{ab}}(\widehat{G},C(X,\mathbb{T}))\longrightarrow H^2(\widehat{G},C(X,\mathbb{T}))\stackrel{\lambda}{\longrightarrow}\Alt^2(\widehat{G},C(X,\mathbb{T}))\longrightarrow 0
\]is split. Moreover, we write $\pr_{\text{ab}}:H^2(\widehat{G},C(X,\mathbb{T}))\rightarrow H^2_{\text{ab}}(\widehat{G},C(X,\mathbb{T}))$ for the induced projection map.

\begin{prop}\label{comm V}
	The class $\pr_{\emph{ab}}([\mathcal{C}_M\Psi])\in H^2_{\emph{ab}}(\widehat{G},C(X,\mathbb{T}))$ does not depend on the choice of the factor system and is therefore an invariant for the set $\Ext(C(X),G,\varphi)$.
\end{prop}
\begin{proof}
	Let $(M_\pi, \Psi_{\pi, \rho})_{\pi, \rho\in\widehat{G}}$ and $(M'_\pi, \Psi'_{\pi, \rho})_{\pi, \rho\in\widehat{G}}$ be two factor systems for the map $\varphi$. Then it follows from Lemma \ref{comm III} and the construction of the map $\pr_{\text{ab}}$ that
	\begin{align*}
		\pr_{\text{ab}}([\mathcal{C}_{M'}\Psi'])
		&=\pr_{\text{ab}}([\lambda_{\omega}\mathcal{C}_M\Psi])
		=\pr_{\text{ab}}([\lambda_{\omega}][\mathcal{C}_M\Psi])
		=\pr_{\text{ab}}([\mathcal{C}_M\Psi]).
		\qedhere
	\end{align*}
\end{proof}

\begin{defn}\label{comm VI}
Let $\varphi:\widehat{G}\rightarrow\Pic(X)$ be a group homomorphism. Then we call
\begin{align*}
\chi_2(\varphi):=\pr_{\text{ab}}([\mathcal{C}_M\Psi])\in H^2_{\text{ab}}(\widehat{G},C(X,\mathbb{T}))
\end{align*}
the \emph{secondary characteristic class} of $\varphi$.
\end{defn}

The next result provides a group theoretic criterion for the existence of topological principal $G$-bundle over $X$.

\begin{thm}\label{comm VII}
Let $G$ be a compact Abelian group and $X$ a compact space. Furthermore, let $\varphi:\widehat{G}\rightarrow\Pic(X)$ be a group homomorphism. Then the following statements are equivalent:
\begin{itemize}
\item[\emph{(}a\emph{)}]
The set $\Ext_{\emph{\text{top}}}(X,G,\varphi)$ is non-empty, that is,
there exists a topological principal bundle $\sigma:P\times G\rightarrow P$ over $X$ representing an element of $\Ext_{\emph{\text{top}}}(X,G,\varphi)$.
\item[\emph{(}b\emph{)}]
The map $\varphi$ satisfies the following two conditions in the indicated order:
\begin{itemize}
\item[\emph{(}$b_1$\emph{)}]
The class $\chi(\varphi)\in H^3(\widehat{G},C(X,\mathbb{T}))$ vanishes.
\item[\emph{(}$b_2$\emph{)}]
Furthermore, the class $\chi_2(\varphi)\in H^2_{\emph{ab}}(\widehat{G},C(X,\mathbb{T}))$ vanishes.
\end{itemize}
\end{itemize}
\end{thm} 
\begin{proof}
Suppose first that the set $\Ext_{\text{top}}(X,G,\varphi)$ is non-empty. Then Theorem~\ref{classsatactgrpcomm} implies that the characteristic class $\chi(\varphi)$ in $H^3(\widehat{G},C(X,\mathbb{T}))$ vanishes. To verify that the class $\chi_2(\varphi)\in H^2_{\text{ab}}(\widehat{G},C(X,\mathbb{T}))$ vanishes, let $\sigma:P\times G\rightarrow P$ be a topological principal bundle over $X$ representing an element of $\Ext_{\text{top}}(X,G,\varphi)$. Then the canonical factor system of the associated free C\Star dynamical system $(C(P),G,\alpha_{\sigma})$ (cf. Theorem \ref{sat=free/comm}) is given by
\begin{align*}
\Psi_{\pi, \rho}:C(P)(\pi)\widehat{\otimes}_{C(X)}C(P)(\rho)\rightarrow C(P)(\pi+\rho), \quad f\otimes_{C(X)}g\mapsto fg.
\end{align*}
Therefore, the claim follows from the commutativity of $C(P)$ since we have 
\begin{align*}
\Psi_{\pi, \rho}(f\otimes_{C(X)}g)=fg=gf=\Psi_{\rho,\pi}(g\otimes_{C(X)}f)
\end{align*}
for all $f\in C(P)(\pi)$ and $g\in C(P)(\rho)$, \ie, $\mathcal{C}_M\Psi=\one_{C(X)}$.

If, conversely, condition $(b_1)$ is satisfied, then it follows from Theorem \ref{classsatactgrpcomm} that there is a free C\Star dynamical system $(\aA,G,\alpha)$ representing an element of $\Ext(C(X),G,\varphi)$. Let $(M_\pi, \Psi_{\pi, \rho})_{\pi, \rho\in\widehat{G}}$ be its associated factor system. We then use condition $(b_2)$ to find an element $\omega\in Z^2(\widehat{G},C(X,\mathbb{T}))$ such that 
\begin{align*}
\lambda_{\omega^*}=\mathcal{C}_M\Psi\in\Alt^2(\widehat{G},C(X,\mathbb{T})).
\end{align*}
Consequently, the corresponding factor system $(M_\pi, \omega(\pi,\rho)\Psi_{\pi, \rho})_{\pi, \rho\in\widehat{G}}$ has the property $\mathcal{C}_M(\omega\Psi)=\one_{C(X)}$, that is, its associated free C\Star dynamical system is equivalent to one of the form $(C(P),G,\alpha_{\sigma})$ induced by some topological principal bundle $\sigma:P\times G\rightarrow P$ over~$X$.
\end{proof}

The following statement provides a classification of topological principal $G$-bundles over~$X$. It is a consequence of Corollary \ref{sat and cocycles cor} and Lemma \ref{comm III}. 

\begin{cor}\label{class top principal bundles}
    Let $G$ be a compact Abelian group and $X$ a compact space. Furthermore, let $\varphi:\widehat{G}\rightarrow\Pic(X)$ be a group homomorphism with $\Ext_{\emph{\text{top}}}(X,G,\varphi)\neq\emptyset$. Then the map
	\begin{align*}
	H^2_{\emph{\text{ab}}}(\widehat{G},C(X,\mathbb{T}))\times\Ext_{\emph{\text{top}}}(X,G,\varphi)&\rightarrow\Ext_{\emph{\text{top}}}(X,G,\varphi), \quad \Bigr(\bigr[\omega\bigl],\bigr[P_{(M,\Psi)}\bigl]\Bigl) \mapsto \bigr[P_{(M,\omega\Psi)}\bigl]
	\end{align*}
	is a well-defined simply transitive action, where $P_{(M,\Psi)}$ denotes the topological principal bundle associated to a given factor system $(M,\Psi):=(M_\pi, \Psi_{\pi, \rho})_{\pi, \rho\in\widehat{G}}$. 
\end{cor}

We conclude this section with a few words on how our previous results relates to the \v Cech cohomology for the pair $(X,G)$ classifying locally trivial principal $G$-bundles over~$X$. 
\begin{rmk}
	It a well-known fact that locally trivial principal $G$-bundles over $X$ are up to equivalence (\cf Definition~\ref{def: top bundles}) classified by the \v Cech cohomology group $\check H^1(X,G)$. It follows that for a compact Abelian group $G$ we have a canonical injection
	\begin{align*}
		\check H^1(X,G) \hookrightarrow \Ext_{\text{top}}(X,G):=\dot{\bigcup_{\varphi}}\Ext_{\text{top}}(X,G,\varphi),
	\end{align*}
	where the disjoint union is taken over all group homomorphisms $\varphi:\widehat G \to \Pic(X)$. Unfortunately, we do not know yet whether the action of $H^2_{\text{ab}} ( \widehat{G},C(X,\mathbb{T}))$ of Corollary~\ref{class top principal bundles} can be pulled back to an action on $\check{H}^1(X,G)$. 
\end{rmk}

\section{Examples}
\label{sec:examples}

In the last section of this paper we present some examples.

\begin{expl}
Let $G$ be a compact Abelian group. Furthermore, let $\mathbb{M}_m(\mathbb{C})$ be the C\Star algebra of $m\times m$ matrices and recall that its natural representation on $\mathbb{C}^m$ is, up to equivalence, the only irreducible representation of $\mathbb{M}_m(\mathbb{C})$. Therefore, it follows from Example \ref{expl:pic_grp} that $\Pic(\mathbb{M}_m(\mathbb{C}))$ is trivial. In particular, there is only the trivial group homomorphism from $\widehat{G}$ to $\Pic(\mathbb{M}_m(\mathbb{C}))$ and a realization is given by the free C\Star dynamical system $$(C(G,\mathbb{M}_m(\mathbb{C})),G,\text{rt}\otimes\id_{\mathbb{M}_n(\mathbb{C})}),$$ where
\begin{align*}
\text{rt}:G\times C(G)\rightarrow C(G), \quad \text{rt}(g,f)(h):=f(hg)
\end{align*}
denotes the right-translation action by $G$. Moreover, we conclude from Corollary \ref{sat and cocycles cor} that all free actions of $G$ with fixed point algebra $\mathbb{M}_m(\mathbb{C})$ are parametrized by the cohomology group $H^2(\widehat{G},\mathbb{T})$. In the case $G=\mathbb{T}^n$, $n\in\mathbb{N}$, this cohomology group is isomorphic to $\mathbb{T}^{\frac{1}{2}n(n-1)}$ and parametrizes the free actions given by tensor products of the noncommutative $n$-tori endowed with their natural $\mathbb{T}^n$-action (\cf Example \ref{expl sat act a}) and the C\Star algebra $\mathbb{M}_m(\mathbb{C})$.
\end{expl}

\begin{expl}
	Consider the 2-fold direct sum $\mathbb{M}_2(\mathbb{C})\oplus\mathbb{M}_2(\mathbb{C})$ and notice that the group $UZ(\mathbb{M}_2(\mathbb{C})\oplus\mathbb{M}_2(\mathbb{C}))$ is isomorphic to $\mathbb{T}^2$. Since the spectrum of $\mathbb{M}_2(\mathbb{C})\oplus\mathbb{M}_2(\mathbb{C})$ contains two elements, it follows from Example \ref{expl:pic_grp} that $\Pic(\mathbb{M}_2(\mathbb{C})\oplus\mathbb{M}_2(\mathbb{C}))$ is isomorphic to~$\mathbb{Z}_2$.  If $\varphi:\mathbb{Z}\rightarrow\mathbb{Z}_2$ denote the canonical group homomorphism with kernel $2\mathbb{Z}$, then it is a consequence of \cite[Chapter VI.6]{Ma95} that the cohomology groups $H^2_{\varphi}(\mathbb{Z},\mathbb{T}^2)$ and $H^3_{\varphi}(\mathbb{Z},\mathbb{T}^2)$ are trivial. Therefore, Theorem \ref{classsatactgrpcomm} implies that the set $\Ext(\mathbb{M}_2(\mathbb{C})\oplus\mathbb{M}_2(\mathbb{C}),\mathbb{T},\varphi)$ is non-empty and contains according to Corollary \ref{sat and cocycles cor} exactly one element, namely the class of the trivial system $$(C(\mathbb{T},\mathbb{M}_2(\mathbb{C})\oplus\mathbb{M}_2(\mathbb{C}),\mathbb{T},\text{rt}\otimes\id).$$
\end{expl}

\begin{expl}\label{cool quantum torus example}
For the following discussion we recall the notation from Example \ref{expl sat act a}. Let $\mathbb{T}^n_{\theta}$ be the noncommutative $n$-torus defined by the real skew-symmetric $n\times n$ matrix~$\theta$ and let $\omega_1$ be the corresponding $\mathbb{T}$-valued 2-cocycle on $\mathbb{Z}^n$ given for ${\mathbf k},{\mathbf k'}\in\mathbb{Z}^n$ by 
\begin{align*}
\omega_1({\mathbf k},{\mathbf k'}):=\exp({}_{\mathbb{C}^n}\langle\theta{\mathbf k},{\mathbf k'}\rangle).
\end{align*}
Furthermore, let $S:\mathbb{Z}^m\rightarrow\Aut(\mathbb{T}^n_{\theta})$ be a group homomorphism leaving the isotypic components of $\mathbb{T}^n_{\theta}$ (with respect to the canonical gauge action by $\mathbb{T}^n$) invariant, \ie, such that for all $ {\mathbf l}\in\mathbb{Z}^m$ and ${\mathbf k}\in\mathbb{Z}^n$
\begin{align*}
S({\mathbf l})U_{\mathbf k}=c_{{\mathbf l},{\mathbf k}}U_{\mathbf k}
\end{align*}
for some $c_{{\mathbf l},{\mathbf k}}\in\mathbb{T}$. Then, given another $\mathbb{T}$-valued 2-cocycle $\omega_2$ on $\mathbb{Z}^m$, it follows from Remark \ref{factor systems TNCPB}, that the pair $(S,\omega_2)$ gives rise to a factor system for the group homomorphism $\varphi:=\pr_{\mathbb{T}^n_{\theta}}\circ S:\mathbb{Z}^m\rightarrow\Pic(\mathbb{T}^n_{\theta})$. Moreover, it is easily seen that the associated free C\Star dynamical system is equivalent to the free C\Star dynamical system $(\mathbb{T}^{n+m}_{\theta'},\mathbb{T}^m,\alpha)$, where $\mathbb{T}^{n+m}_{\theta'}$ denotes the noncommutative $(n+m)$-torus determined by the $\mathbb{T}$-valued 2-cocycle on $\mathbb{Z}^{n+m}$ given for ${\mathbf k},{\mathbf k'}\in\mathbb{Z}^n$ and ${\mathbf l},{\mathbf l'}\in\mathbb{Z}^m$ by
\begin{align*}
\omega\bigr(({\mathbf k},{\mathbf l}),({\mathbf k'},{\mathbf l'})\bigl):=c_{{\mathbf l},{\mathbf k'}}\omega_1({\mathbf k},{\mathbf k'})\omega_2({\mathbf l},{\mathbf l'})
\end{align*}
and $\alpha$ is the restriction of the gauge action $\alpha^{n+m}_{\theta'}$ to the closed subgroup $\mathbb{T}^m$ of $\mathbb{T}^{n+m}$. That $(\mathbb{T}^{n+m}_{\theta'},\mathbb{T}^m,\alpha)$ is actually free is a consequence of Proposition \ref{restricted free}. In particular, it represents an element in $\Ext(\mathbb{T}^n_{\theta},\mathbb{T}^m,\varphi)$ which is, according to Corollary \ref{sat and cocycles cor}, parametrized by the cohomology group 
\begin{align*}
H^2_\varphi(\mathbb{Z}^m,UZ(\mathbb{T}^n_{\theta})).
\end{align*}
\end{expl}

\begin{rmk}
The previous example can be used to construct a free C\Star dynamical system with commutative fixed point algebra which is not a C\Star algebraic bundle of C\Star dynamical systems over the spectrum of the fixed point algebra in a canonical way (\cf \cite{ENOO09,ReWi98}). Indeed, if we simply chose $\theta=0$ and a nontrivial group homomorphism $S:\mathbb{Z}^m\rightarrow\Aut(C(\mathbb{T}^n))$, then it is easily seen that the fixed point algebra $C(\mathbb{T}^n)$ is not contained in the center of $\mathbb{T}^{n+m}_{\theta'}$. 
\end{rmk}

\begin{expl}
Let $\theta$ be an irrational number in $[0,1]$ and $\mathbb{T}^2_{\theta}$ the corresponding noncommutative $2$-torus from Example \ref{expl sat act a}. We recall that in this case $UZ(\mathbb{T}^2_{\theta})$ is isomorphic to $\mathbb{T}$. Furthermore, let $\varphi:\mathbb{Z}^2\rightarrow\Pic(\mathbb{T}^2_{\theta})$ be any group homomorphism (note that $\mathbb{T}^2\subseteq\Aut(\mathbb{T}^2_{\theta})$ to apply the construction in Example \ref{cool quantum torus example}). Then it is a consequence of \cite[Chapter VI.6]{Ma95} that the cohomology group $H^3_{\varphi}(\mathbb{Z}^2,\mathbb{T})$ is trivial. Therefore, Theorem~\ref{classsatactgrpcomm} implies that the set $\Ext(\mathbb{T}^2_{\theta},\mathbb{T}^2,\varphi)$ is non-empty and, according to Corollary~\ref{sat and cocycles cor}, parametrized by the cohomology group $H^2_{\varphi}(\mathbb{Z}^2,\mathbb{T})$. For its computation we refer, for example, to \cite[Proposition 6.2]{Wa11b}.
\end{expl}

\begin{expl}
	Let $H$ be the discrete (three-dimensional) Heisenberg group and let $(C^{*}(H),\mathbb{T}^2,\alpha)$ the corresponding free C\Star dynamical system from Example \ref{expl sat act b}. If 
	\begin{align*}
	\varphi:\mathbb{Z}^2\rightarrow\Pic(C(\mathbb{T}))\cong\Pic(\mathbb{T})\rtimes\Homeo(\mathbb{T})
	\end{align*}
	denotes the associated group homomorphism, then the class of $(C^{*}(H),\mathbb{T}^2,\alpha)$ is contained in the set $\Ext(C(\mathbb{T}),\mathbb{T}^2,\varphi)$ of equivalence classes of realizations of $\varphi$, which is by Corollary \ref{sat and cocycles cor} parametrized by the cohomology group $H^2_{\varphi}(\mathbb{Z}^2,C(\mathbb{T},\mathbb{T}))$. For its computation we refer, again, to \cite[Proposition 6.2]{Wa11b}.
\end{expl}

\begin{expl}
For $q\in[-1,1]$ let $(\SU_q(2),\mathbb{T},\alpha)$ be the quantum Hopf fibration from Example \ref{expl sat act c} and $L_q(1)$ the isotypic component corresponding to $1\in\mathbb{Z}$. If 
\begin{align*}
\varphi:\mathbb{Z}\rightarrow\Pic(S_q(2)), \quad 1\mapsto[L_q(1)]
\end{align*}
denotes the associated group homomorphism, then the class of $(\SU_q(2),\mathbb{T},\alpha)$ is contained in the set $\Ext(S_q(2),\mathbb{T},\varphi)$ of equivalence classes of realizations of $\varphi$, which is by Corollary \ref{sat and cocycles cor} parametrized by the cohomology group $H^2_{\varphi}(\mathbb{Z},UZ(S_q(2))$. It follows, for example, from \cite[Chapter IV.6]{Ma95} that this cohomology group is trivial, \ie, up to isomorphism the quantum Hopf fibration $(\SU_q(2),\mathbb{T},\alpha)$ is the unique realization of the group homomorphism $\varphi$.
\end{expl}

\section*{Acknowledgement}

We thank the Danish National Research Foundation through the Center for Symmetry and Deformation (DNRF92). We also thank the Academy of Finland through the Center of Excellence in Analysis and Dynamics Research and the Research Grant 1138810. Furthermore, we thank T. Crisp, E. Meir and R. Nest for stimulating discussions on the subject matter of this paper. 




\begin{thebibliography}{xxxxxxxxxx}

\bibitem{AlHo80}
Albeverio, S. and Raphael H\o egh--Krohn
\emph{Ergodic Actions by Compact Groups on C\Star Algebras}, 
Math. Z. 174, 1-17, (1980).

\bibitem{AAGBL94}
\'{A}lvarez, E., L. \'{A}lvarez-Gaum\'{e},  J. L. F. Barb\'{o}n and Y. Lozano,
\emph{Some Global Aspects of Duality in String Theory},
Nucl. Phys. B415, 71-100 (1994).

\bibitem{Amm}
Ammann, B.,
\emph{The Dirac Operator on Collapsing $\mathbb{S}^1$-Bundles}, 
Sem. Th. Spec. Geom Inst. Fourier Grenoble 16, 33-42, (1998).

\bibitem{AmmB}
Ammann, B. and C. B\"ar,
\emph{The Dirac Operator on Nilmanifolds and Collapsing Circle Bundles}, 
Ann. Global Anal. Geom. 16, no. 3, 221-253, (1998).

\bibitem{BHS07}
Baum, P., P. Hajac, R. Matthes and W. Szymanski,
\emph{Non-Commutative Geometry Approach to Principal and Associated Bundles}, 
in Quantum Symmetry in Noncommutative Geometry, to appear, arXiv:0701033v2 [math.DG], 8 Jan 2007.

\bibitem{BDH13}
Baum, P., K. De Commer and P. Hajac, 
\emph{Free Actions of Compact Quantum Groups on Unital C\Star Algebras}, 
arXiv:1304.2812v4 [math.OA], 30 Jun 2015.

\bibitem{BeCoLe11}
Bertozzini, P., R. Conti and W. Lewkeeratiyutkul,
\emph{Non-Commutative Geometry, Categories and Quantum Physics},
East-West Journal of Mathematics ``Contributions in Mathematics and Applications II'' Special Volume 2007, 213-259 (2008).

\bibitem{BiRiVa06}
J. Bichon, A. De Rijdt and S. Vaes, 
\emph{Ergodic coactions with large multiplicity and monodial equivalence of quantum groups}, 
Comm. Math. Phys. 262 (2006), 703-728.

\bibitem{Bl98}
Blackadar, B.,
\emph{K-Theory for Operator Algebras}, 
Mathematical Sciences Research Institute Publications, Springer, New York, 1998.

\bibitem{BrGrRi77}
Brown, L. G., P. Green and M. A. Rieffel,
\emph{Stable Isomorphism and Strong Morita Equivalence of C\Star Algebras }, 
Pacific J. Math., Vol. 71, No. 2, (1977).

\bibitem{BuWe04}
Bursztyn, H. and A. Weinstein,
\emph{Picard Groups in Poisson Geometry},
Mosc. Math. J., Vol. 4,	No 2, 39-66, (2004).

\bibitem{Bu87}
Buscher, T.,
\emph{Asymmetry of the String Back Ground Field Equations},
Phys. Lett. B194, 59-62 (1987).



\bibitem{DaSi}
Dabrowski, L. and A. Sitarz,
\emph{Noncommutative Circle Bundles and new Dirac Operators}, 
Comm. Math. Phys., 318, 111-130, (2013).

\bibitem{DaSiZu}
Dabrowski, L., A. Sitarz and A. Zucca,
\emph{Dirac Operators on Noncommutative Principal Circle Bundles}, 
Int. J. Geom. Methods Mod. Phys. 11, 1450012 (2014).

\bibitem{DKY12}
De Commer, K. and M. Yamashita,
\emph{A Construction of Finite Index C\Star Algebra Inclusions From Free Actions of Compact Quantum Groups}, 
Publ. Res. Inst. Math. Sci. 49 (2013), 709-735. 

\bibitem{Do80}
Douglas, R.,
\emph{C\Star Algebra Extensions and K-Homology}, 
Ann. of Math. Studies, Vol. 95, Princeton, NJ: Princeton 1980.

\bibitem{EKQR06}
Echterhoff, S., S. Kaliszewski, J. Quigg and I. Raeburn, 
\emph{A Categorical Approach to Imprimitivity Theorems for C\Star Dynamical Systems}, in 
Memoirs of the American Mathematical Society, 2006.

\bibitem{ENOO09}
Echterhoff, S., R. Nest and H. Oyono-Oyono, 
\emph{Principal Non-Commutative Torus Bundles}, in 
Proc. London Math. Soc., (3) 99 (2009), 1-31.

\bibitem{Ellwood00}
Ellwood, D. A.,
\emph{A New Characterization of Principal Actions},  
J. Funct. Anal., 173, (2000), 49-60.


\bibitem{GoLaPe94}
Gootman, E. C., A. J. Lazar and C. Peligrad,
\emph{Spectra for Compact Group Actions}, 
J. Operator Theory 31 (1994), 381-399.

\bibitem{Green77}
Green, P., 
\emph{C\Star Algebras of Transformation Groups With Smooth Orbit Space}, 
Pacific J. Math. 72 (1977), no. 1, 71-97.


\bibitem{Haj04}
Hajac, P. M.,
\emph{Lecture Notes on Noncommutative Geometry and Quantum Groups}, 
Lecture Notes edited by Piotr M. Hajac, \url{http://www.mimuw.edu.pl/~pwit/toknotes/toknotes.pdf}.

\bibitem{HaMa10}
Hannabuss, K. and V. Mathai,
\emph{Noncommutative Principal Torus Bundles via Parametrised Strict Deformation Quantization}, in Proc. Symp. Pure Math. 81, (2010), 133-148.

\bibitem{HoMo06}
Hofmann, K.H. and S.A. Morris, 
\emph{The Structure of Compact Groups}, 2nd Revised and Augmented Edition,
de Gruyter Studies in Mathematics 25, Cambridge University Press, 2006.

\bibitem{Hu75}
Husemoller, D.,
\emph{Fibre Bundles}, 
Graduate Texts in Mathematics, vol. 20, Springer-Verlag, New York, 1975.

\bibitem{Ko97}
Kodaka, K.,
\emph{Picard Groups of Irrational Rotation C\Star Algebras}, 
Proc. London Math. Soc. (2), 56, (1997), 179-188.

\bibitem{KoNo63}
Kobayashi, S., and K. Nomizu,
\emph{Foundations of Differential Geometry vol.~1}, 
Interscience Tracts in Pure and Appl. Math., Wiley, 1963.

\bibitem{La95}
Lance, E.C.,
\emph{Hilbert C\Star Modules. A Toolkit for Operator Algebraists},
London Mathematical Society Lecture Notes Series, 210. Cambridge University Press, Cambridge, 1995. x+130 pp.

\bibitem{LaSu05}
Landi, G., and W. van Suijlekom,  
\emph{Principal Fibrations From Noncommutative Spheres},
Commun. Math. Phys. 260 (2005), 203-225.

\bibitem{LePr15}
Lentner, S. and J. Priel, 
\emph{A Double Coset Decomposition of the Brauer Picard Group of the Representation Category of a Finite Group}
ZMP-HH 15/9 Hamburger Beitr\"age zur Mathematik Nr. 539.

\bibitem{Ma95}
Mac Lane, S.,  
\emph{Homology}, 
Classics in Mathematics, Springer-Verlag, 1995.

\bibitem{MaRo05}
Mathai, V, and J. Rosenberg,
\emph{T-duality for Torus Bundles With H-Fluxes via Noncommutative Topology},
Comm. Math. Phys. 253 (2005), no. 3, 705-721.

\bibitem{MaRo06}
\textemdash, 
\emph{T-Duality for Torus Bundles With H-Fluxes via Noncommutative Topology. II. The High-Dimensional Case and the T-Duality Group},
Adv. Theor. Math. Phys. 10 (2006), no. 1, 123-158.




\bibitem{Ne06}
Neeb, K.-H.,  
\emph{Non-Abelian Extensions of Infinite-Dimensional Lie-Groups}, 
Annales de l'Inst. Fourier 57, (2007), 209-271.

\bibitem{Ne07}
\textemdash, 
\emph{On the Classification of Rational Quantum Tori and the Structure of Their Automorphism Groups}, 
Canad. Math. Bull. 51(2008), 261-282.

\bibitem{Ne08b}
\textemdash, 
\emph{Differential Topology of Fiber Bundles}, 
Lecture Notes, Darmstadt, 2008.

\bibitem{Ne13}
Neshveyev, S., 
\emph{Duality Theory for Nonergodic Actions},
M\"unster J. of Math. 7(2):414-437, 2013.

\bibitem{OlPeTa80}
Olesen, D., G. K. Pedersen and M. Takesaki,
\emph{Ergodic Actions of Compact Abelian Groups},
J. Operator Theory, 3 (1980), 237-269.

\bibitem{Ph87}
Phillips, N. C.,
\emph{Equivariant K-Theory and Freeness of Group Actions on C\Star Algebras}, 
Springer-Verlag Lecture Notes in Math. no. 1274, Springer-Verlag, Berlin, Heidelberg, New York, London, Paris, Tokyo, 1987.

\bibitem{Ph09}
\textemdash,
\emph{Freeness of Actions of Finite Groups on C\Star Algebras}, 
Operator structures and dynamical systems, Contemp. Math., 503, Amer. Math. Soc., Providence, RI, (2009), 217-257.

\bibitem{ReWi98}
Raeburn, I. and D. P. Williams,
\emph{Morita Equivalence and Continuous-Trace C\Star Algebras}, 
Mathematical Surveys and Monographs Volume 60, Amer. Math. Soc., 1998.

\bibitem{Ri88}
Rieffel, M.,
\emph{Proper Actions of Groups on C\Star algebras}, 
Mappings of operator algebras, 141-182, Progr. Math. 84, Birkh\"auser, Boston, MA (1990).

\bibitem{Ro04}
R\o rdam, M.,
\emph{Stable C\Star Algebras}, 
Advanced Studies in Pure Mathematics 38 ``Operator Algebras and Applications''. (2004), 177-199. 

\bibitem{Sch04}
Schauenburg, P.,
\emph{Hopf-Galois and Bi-Galois Extensions}, 
Fields Inst. Commun., 43, Amer. Math. Soc., Providence, RI, (2004), 469-515.

\bibitem{SchWa15}
Schwieger, K. and S. Wagner,
\emph{Part II, Free Actions of Compact Groups on C$^*$-algebras}, 
J. Noncommut. Geom. 11 (2017), 641-668.

\bibitem{SchWa16a}
\textemdash,
\emph{Part III, Free Actions of Compact Quantum Groups on C$^*$-algebras}, 
arxiv:1701.05895v1 [math.OA], 20 Jan 2017

\bibitem{SchWa16b}
\textemdash,
\emph{Towards a Fundamental Group of Noncommutative Spaces}, 
in preparation (2017).

\bibitem{Swa62}
Swan, R. G.,
\emph{Vector Bundles and Projective Modules}, 
Trans. Amer. Math. Soc. 105, (1962), 264-277.

\bibitem{Szy03}
Szyma\'{n}ski, W.,
\emph{Quantum Lens Spaces and Principal Actions on Graph C\Star Algebras}, 
arXiv:0209400v1 [math.QA], 29 Sep 2003.

\bibitem{To00}
Tom Dieck, T.,
\emph{Topologie}, 
de Gruyter, 2. Auflage, 2000.

\bibitem{Wa11a}
Wagner, S.,
\emph{A Geometric Approach to Noncommutative Principal Bundles}, PhD-Thesis
arXiv:1108.0311v1 [math.DG], 1 Aug 2011.

\bibitem{Wa11b}
\textemdash, 
\emph{Trivial Noncommutative Principal Torus Bundles}, 
Noncommutative Harmonic Analysis, Banach Center Publ. 96, (2012), 299-317.


\bibitem{Wa11d}
\textemdash,
\emph{A Geometric Approach to Noncommutative Principal Torus Bundles}, 
Proc. London Math. Soc. (2013) 106 (6): 1179-1222.

\bibitem{Wa11e}
\textemdash,
\emph{On Noncommutative Principal Bundles with Finite Abelian Structure Group}, 
J. Noncommut. Geom. 8 (2014), 987-1022. 


\bibitem{Wa10}
Wahl, C.,
\emph{Index Theory for Actions of Compact Lie Groups on C\Star Algebras}, 
J. Oper. Theory, 63 (1), (2012), 216-242.

\bibitem{Wass88a}
Wassermann, A.,
\emph{Ergodic Actions of Compact Groups on Operator Algebras II: Classification of Full Multiplicity Ergodic Actions}, 
Can. J. Math., Vol. XL, No. 6 (1988), pp. 1482-1527.

\bibitem{Wass88b}
\textemdash,  
\emph{Ergodic Actions of Compact Groups on Operator Algebras. {III}. {Classification} for {SU(2)}},
Invent. Math., 93 (2), 1988, pp.~309--354

\bibitem{Wass89}
\textemdash,  
\emph{Ergodic Actions of Compact Groups on Operator Algebras I. General Theory}, 
Annals of Math., Vol. 130, No. 2 (Sep., 1989), pp. 273-319.

\bibitem{Wo87}
Woronowicz, S. L.,
\emph{Twisted SU\emph{(}2\emph{)} Group. An Example of a Noncommutative Differential Calculus}, 
Publ. RIMS, Kyoto Univ., Vol. 23, 1 (1987) 117 23 (1), 117-181.

\end{thebibliography}
\end{document}